\documentclass[french]{article} 
\usepackage[utf8]{inputenc}
\usepackage[T1]{fontenc}
\usepackage{lmodern}
\usepackage[a4paper]{geometry}
\usepackage[english]{babel}

\usepackage{fullpage}

\usepackage{amssymb,amsmath,mathtools,amsthm,dsfont}
 \usepackage[a]{esvect}

\usepackage[all]{xy}
\usepackage{xcolor}

\newtheorem{theorem}{Theorem}
\newtheorem{prop}[theorem]{Proposition}

\newtheorem{lemma}[theorem]{Lemma}

\newtheorem{rem}[theorem]{Remark}

\newcommand{\E}{\mathbb{E}}

\newcommand{\N}{\mathbb{N}}

\renewcommand{\P}{\mathbb{P}}

\newcommand{\R}{\mathbb{R}}

\newcommand{\Z}{\mathbb{Z}}

\newcommand{\cC}{\mathcal{C}}
\newcommand{\cD}{\mathcal{D}}

\newcommand{\cM}{\mathcal{M}}

\newcommand{\cS}{\mathcal{S}}

\renewcommand{\1}{\mathds{1}}
\renewcommand{\d}{\text{d}}

\renewcommand{\epsilon}{\varepsilon}
\renewcommand{\phi}{\varphi}

\newcounter{numeroexo}

\newcommand{\projv}{\text{Proj}^{d_2}}
\newcommand{\projh}{\text{Proj}_{d_1}}

\newcommand{\bT}{\mathbf{T}}
\newcommand{\bmu}{\pmb{\mu}}
\newcommand{\btau}{\pmb{\tau}}
\newcommand{\bC}{\mathbf{C}}
\newcommand{\bcD}{\pmb{\mathcal{D}}}
\newcommand{\bV}{\mathbf{V}}
\newcommand{\bR}{\mathbf{R}}
\newcommand{\bS}{\mathbf{S}}
\newcommand{\bsigma}{\pmb{\sigma}}
\newcommand{\bcC}{\pmb{\mathcal{C}}}
\newcommand{\bcM}{\pmb{\mathcal{M}}}
\newcommand{\bw}{\mathbf{w}}
\newcommand{\bW}{\mathbf{W}}

\title{\LARGE First-order behavior of the time constant in Bernoulli first-passage percolation}
\author{Anne-Laure Basdevant\footnote{Modal'X, UPL, Universit\'e Paris Nanterre, 92000 Nanterre, France, and FP2M, CNRS FR 2036, {\it anne.laure.basdevant@normalesup.org}}, Jean-Baptiste Gou\'er\'e\footnote{Institut Denis-Poisson - UMR CNRS 7013, Universit\'e de Tours, Parc de Grandmont, 37200 Tours, France, {\it jean-baptiste.gouere@lmpt.univ-tours.fr}} ~and Marie Th\'eret\footnote{Modal'X, UPL, Universit\'e Paris Nanterre, 92000 Nanterre, France, and FP2M, CNRS FR 2036, {\it marie.theret@parisnanterre.fr}}}
\date{}

\begin{document}

\selectlanguage{english}

\maketitle

\begin{abstract}
We consider the standard model of first-passage percolation on $\mathbb{Z}^d$ ($d\geq 2$), with i.i.d. passage times associated with either the edges or the vertices of the graph. We focus on the particular case where the distribution of the passage times is the Bernoulli distribution with parameter $1-\epsilon$. These passage times induce a random pseudo-metric $T_\epsilon$ on $\mathbb{R}^d$. By subadditive arguments, it is well known that for any $z\in\mathbb{R}^d\setminus \{0\}$, the sequence $T_\epsilon (0,\lfloor nz \rfloor) / n$ converges a.s. towards a constant $\mu_\epsilon (z)$ called the time constant. We investigate the behavior of $\epsilon \mapsto \mu_\epsilon (z)$ near $0$, and prove that $\mu_\epsilon (z) =  \| z\|_1 - C (z) \epsilon ^{1/d_1(z)} + o ( \epsilon ^{1/d_1(z)}) $, where $d_1(z)$ is the number of non null coordinates of $z$, and $C(z)$ is a constant whose dependence on $z$ is partially explicit.
\end{abstract}

\section{Introduction and main results}

\label{s:intro}

\paragraph{Historic of first-passage percolation.} The model of first-passage percolation has been introduced by Hammersley and Welsh in the seminal paper \cite{Hammersley-Welsh} as a refinement of percolation to model propagation phenomena : instead of wondering {\em if} the propagation occurs, the question this model aims to answer is {\em when} it will occur. We refer to \cite{AuffingerDamronHanson,Kesten-saint-flour} for surveys on the subject. 

In the classical model of first-passage percolation on $\mathbb{Z}^d$, a non-negative random variable is associated with every edge of the graph. It is called the passage time of the edge, and represents the time needed to cross the edge. In what follows, we consider a particular case of this model in which the passage times have a Bernoulli distribution. We consider also a variant of the model in which the passage times are associated with the vertices of the graph instead of the edges - exactly as site percolation is a variant of bond percolation. This site first-passage percolation model is not classically studied in the literature, even if it is as natural as its bond version, but it appears to be easier to handle in our context. We now start giving precise definitions of the objects of interest.

\paragraph{Bernoulli bond first-passage percolation on $\Z^d$.} 
Let $d \ge 2$ and $\epsilon \in [0,1]$.
We consider on $\Z^d$ the usual graph structure: two vertices $x, y \in \Z^d$ are neighbors if the Euclidean distance between $x$ and $y$ is one. We denote by $\mathbb{E}^d$ the set of edges between neighbors.
Let $(\btau_\epsilon(u))_{u \in \E^d}$ be a family of independent Bernoulli random variables with parameter $1-\epsilon$.

A path $\pi =(x_0,u_1, x_1,\dots , u_n,x_n)$ is an alternating sequence of vertices $(x_0,\dots,x_n)$ and edges $(u_1,\dots,u_n)$ such that for any $i \in \{1,\dots,n\}$, $x_{i-1}$ and $x_i$ are neighbors and $u_i$ denotes the edge with endpoints $x_{i-1}$ and $x_i$. Notice that such a path $\pi$ is entirely described by its vertices or by its edges, thus for short we write $\pi = (x_0,\dots,x_n)$ or $\pi = (u_1,\dots,u_n)$ according to our center of interest. The travel time of such a path is
\[
\btau_\epsilon(\pi) = \sum_{i=1}^{n} \btau_\epsilon(u_i).
\]
If $x$ and $y$ are two vertices of $\Z^d$, then the time between $x$ and $y$ is
\[
\bT_\epsilon(x,y) = \inf_{\pi : x \to y} \btau_\epsilon(\pi)
\]
where the infimum is taken over all paths from $x$ to $y$. The variable $\btau_\epsilon (u)$ is thus seen as the time needed to cross the edge $u$. For that reason, if $\btau_\epsilon (u)=0$, we say that $u$ is open.
For any $z \in \R^d \setminus \{0\}$, there exists a deterministic constant $\bmu_\epsilon(z) \ge 0$ such that
\[
\lim_{n \to \infty} \frac{\bT_\epsilon(0,\lfloor nz \rfloor )}n = \bmu_\epsilon(z) \text{ almost surely and in }L^1
\]
where $\lfloor nz \rfloor$ denotes the coordinate-wise integer part of $nz$.
This is a straightforward consequence of Kingman ergodic subadditive theorem, see for instance \cite{Kesten-saint-flour}.
In the first-passage percolation literature, if the Euclidean norm of $z$ is $1$, $\bmu_\epsilon(z)$ is known as the time constant in the direction $z$. We emphasize the fact that the time constant can be defined in a much more general context, namely with non-negative passage times distributed according
to a general distribution, and most of the following results picked in the literature are in fact proved in this general framework. However we decided to present them in the context of Bernoulli first-passage percolation since it is the framework in which our own results are valid.

\paragraph{First properties of the time constant.} 
Some is known about $\bmu_\epsilon$ but notably not that much. 
The function $\bmu_\epsilon$ has the following properties : absolute homogeneity and convexity on $\mathbb{R}^d$, invariance by the symmetries that preserve the graph $\mathbb{Z}^d$ itself.
The positivity of the time constant is well understood, and it can be proved (see \cite{Kesten-saint-flour}, Theorem 6.1) that
\[
\bmu_\epsilon  \equiv 0 \quad \iff \quad \epsilon \geq p_c(d)
\]
where $p_c(d)$ is the critical parameter of i.i.d.~Bernoulli bond percolation on $\mathbb{Z}^d$. When $\bmu_\epsilon$ is not null, it defines a norm on $\mathbb{R}^d$, and its value can never be explicitly calculated except in the trivial case when $\epsilon=0$:  $\bmu_0 (z) = \|z\|_1$, the $\ell^1$-norm of $z$. The convergence of the rescaled passage times towards $\bmu_\epsilon$ is uniform in all directions, which is equivalent with the celebrated shape theorem (see \cite{Cox-Durrett-shape,Kesten-saint-flour,Richardson}). We define 
\[
\mathbf{B}_\epsilon (t) = \{ x\in \R^d \,:\, \bT_{\varepsilon} (0,\lfloor x \rfloor) \leq t \}
\]
as the set of points in $\mathbb{R}^d$ that can be reached from the origin within time $t\geq 0$, and $\mathcal{B}_{\bmu_\epsilon}$ as the unit ball associated with the norm $\bmu_\epsilon$ for $\epsilon < p_c(d)$. Roughly speaking, the shape theorem states that $\mathbf{B}_\epsilon (t) / t$ converges towards the asymptotic deterministic shape $\mathcal{B}_{\bmu_\epsilon}$ when $t$ goes to infinity. The dependence of $\bmu_\epsilon (z)$ with the direction $z$ is not well understood yet, and the strict convexity of $\mathcal{B}_{\bmu_\epsilon}$ is an important open question (see for instance \cite{AuffingerDamronHanson} Section 2.8).

\paragraph{Properties of $\epsilon \mapsto \bmu_\epsilon (z)$.}
What interests us in this paper is rather the behavior of $\bmu_\epsilon (z)$ as a function of $\epsilon$, for a fixed $z\in \mathbb{R}^d$. It is known that $\epsilon \mapsto \bmu_\epsilon (z)$ is continuous, see \cite{Cox,CoxKesten,Kesten-saint-flour}. Chayes, Chayes and Durrett  \cite{ChayesChayesDurrett} investigate the behavior of $\epsilon \mapsto \bmu_\epsilon ((1,0,\dots, 0))$ when $\epsilon$ goes to $p_c(d)$ in dimension $d=2$, more precisely they prove that the speed of decay of $\bmu_\epsilon ((1,0))$ towards $0$ is polynomial with the same power as the one of the correlation length in the corresponding percolation model. We investigate in the present paper the properties of $\epsilon \mapsto \bmu_\epsilon (z)$ near $0$.

Some bounds on $\bmu_\epsilon (z)$ exist, at least for specific $z$. Let us start with upper bounds. By a comparison with the passage time of one deterministic path of shortest length, it is trivial to obtain that $\bmu_\epsilon (z) \leq \|z\|_1 (1-\epsilon)$ for any $\epsilon$. Another upper bound is available for $d=2$ and $z=(1,1)$. By restricting ourselves in the definition of $\bT_\epsilon (0, (n,n))$ to oriented paths (i.e., paths going only to the North and to the East) from $0$ to $(n,n)$, since those paths are made of $2n$ edges, we obtain that $\bT_\epsilon (0, (n,n)) \leq 2n - L_{(n,n)} $, where $ L_{(n,n)}$ is the maximal number of edges of null passage time that such an oriented path can cross. Forget about the vertical edges of null passage time: $L_{(n,n)} \geq L'_{(n,n)}$, the maximal number of horizontal edges of null passage time that an oriented path from $0$ to $(n,n)$ can cross. This is exactly the celebrated discrete Ulam's problem, originally solved by Seppälaïnen \cite{Sep1, Sep2} and revisited by Basdevant, Enriquez, Gerin and Gouéré \cite{BEGG}: using a discrete variant of Hammersley's lines, they prove that $L'_{(n,n)}/n$ converges a.s. when $n$ goes to infinity towards $2\sqrt{\epsilon (1 - \epsilon)} $. This implies that $\bmu_\varepsilon ((1,1))\leq  2 - 2 \sqrt{\epsilon} + o (\sqrt{\epsilon})$. Their result is more general and can be used to give an upper bound on  $\bmu_\epsilon ((a,b))$ for $a,b \neq 0$, but not directly on $\bmu_\varepsilon ((1,0))$.

On the other hand, for a generic dimension $d\geq 2$, some lower bounds on $\bmu_\epsilon ((1,0,\dots, 0))$ can be found in \cite{Cox,Janson,Kesten-saint-flour,SidoVaresSurgailis}. Notably, Sidoravicius, Surgailis and Vares prove in \cite{SidoVaresSurgailis} that for $d=2$, $\bmu_\epsilon ((1,0)) \geq 1-\sqrt{2 (1-(1-\epsilon)^4)} = 1 - \sqrt{8\epsilon} + o (\sqrt{\epsilon})$. Notice that the upper and lower bounds describe above do not have the same first-order behavior in $\epsilon$.

\paragraph{Main result.} 
Clearly, $\bmu_0(z) = \|z\|_1$, the $\ell^1$-norm of $z$.
In this article we investigate the first-order behavior of $\bmu_\epsilon (z)$ as $\epsilon$ tends to $0$.
This depends on the number of non-zero coordinates of $z$ -- which we denote by $d_1(z)$ -- and on the geometric mean of the absolute
values of the non-zero coordinates -- which we denote by $\gamma(z)$. In other words, writing $z=(z_1,\dots,z_d)$,
\[
d_1(z) =  \#\{ i \in \{1,\dots,d\} : z_i \neq 0\}
\]
and
\begin{equation}\label{e:gamma}
\gamma(z)  = \prod_{i \in \{1,\dots,d\} : z_i \neq 0} |z_i|^{1/d_1(z)}.
\end{equation}
To simplify some notations, we also introduce the number of zero coordinates $d_2(z)$:
\[
d_2(z) = d - d_1(z) =   \#\{ i \in \{1,\dots,d\} : z_i = 0\}.
\]

\begin{theorem} \label{t2}
There exists a family of positive constants $(\bC(d_1,d_2))_{d_1 \ge 1, d_2 \ge 0}$ such that the following holds.
For all $d \ge 2$, for all $z \in \R^d \setminus \{0\}$,
\[
\bmu_\epsilon(z) = \|z\|_1 - \bC(d_1(z),d_2(z)) \gamma(z) \epsilon^{1/d_1(z)} + o(\epsilon^{1/d_1(z)}) \text{ as } \epsilon \to 0.
\]
\end{theorem}

\paragraph{Bernoulli site first-passage percolation on $\Z^d$.} 
A similar result holds for site percolation, and it appears that the proof is more intuitive in this context. We give now the corresponding definitions in the context of site first-passage percolation, and for clarity we replace any notation in bold letters we used in the context of bond percolation ($\btau, \bT, \bmu$) by notations with regular letters ($\tau, T,\mu$) for site percolation.

More precisely, let $d \ge 2$ and $\epsilon \in [0,1]$ and let $(\tau_\epsilon(x))_{x\in   \Z^d}$ be a family of independent Bernoulli random variables with parameter $1-\epsilon$. For a path $\pi$ from $x$ to $y$ with vertices $(x_0,\ldots,x_n)$, define now the travel time of $\pi$ as
\[
\tau_\epsilon(\pi) = \sum_{i=0}^{n-1} \tau_\epsilon(x_i).
\] 
Note that we do not consider $\tau_\epsilon(x_n)$. Define the time between $x$ and $y$ by
\[
T_\epsilon(x,y) = \inf_{\pi : x \to y} \tau_\epsilon(\pi)
\]
where the infimum is taken over all paths from $x$ to $y$. As previously, the variable $\tau_\epsilon (x)$ is seen as the time needed to visit the vertex $x$, thus if $\tau_\epsilon (x)=0$, we say that $x$ is open. 
 
As in the context of bond first-passage percolation, by subadditive arguments, we know that 
for all $z \in \R^d$ there exists a deterministic constant $\mu_\epsilon(z) \ge 0$ such that
\[
\lim_{n \to \infty} \frac{T_\epsilon(0,\lfloor nz \rfloor )}n = \mu_\epsilon(z) \text{ almost surely and in }L^1.
\]
Since the proofs are straightforward adaptations of the ones provided by the literature in the context of bond first-passage percolation, we do not rewrite them in the context of site first-passage percolation. However, for completeness of the paper, we give in Appendix \ref{s:proof-cte} a short proof of the convergence of $( \mathbb{E}(T_\epsilon(0,\lfloor nz \rfloor ))/n)$, since it is enough to define $\mu_\epsilon (z)$ properly as the limit of these rescaled expectations.

We now state the corresponding result on the first-order behavior of the time constant in site first-passage percolation.
\begin{theorem} \label{t}
There exists a family of positive constants $(C(d_1,d_2))_{d_1 \ge 1, d_2 \ge 0}$ such that the following holds. For all $d\geq 2$, for all $z \in \R^d \setminus\{0\}$,
\[
\mu_\epsilon(z) = \|z\|_1 - C (d_1(z),d_2(z)) \gamma(z) \epsilon^{1/d_1(z)} + o(\epsilon^{1/d_1(z)}) \text{ as } \epsilon \to 0.
\]
\end{theorem}

Moreover, we have the following comparison between the constants appearing in the two previous theorems.
\begin{prop}\label{p}
Let $(\bC(d_1,d_2)_{d_1 \ge 1, d_2 \ge 0}$ (resp. $(C(d_1,d_2))_{d_1 \ge 1, d_2 \ge 0}$) be the constants appearing in Theorem \ref{t2} (resp. Theorem \ref{t}), we have
\[
 d_1^{1/d_1}C(d_1,d_2) \le \bC(d_1,d_2) \le (d_1+d_2)^{1/d_1}C(d_1,d_2).
 \]
In particular when $d_2=0$ we get the equality 
\[
\bC(d_1,0)=d_1^{1/d_1}C(d_1,0).
\]
If moreover $d=d_1=2$, the value of both constants is explicit:
\[
C(2,0) = 2 \quad \textrm{and} \quad \bC (2,0) =   2^{3/2}.
\]
\end{prop}

In fact the constants $C(d_1,d_2)$ and $\bC(d_1,d_2)$ have an explicit interpretation in terms of an auxiliary semi-continuous (partially) oriented model, see \eqref{e:expressionC} and \eqref{e:expressionC2}. The case $d_2=0$ corresponds to the diagonal case, in which this second model is in fact totally continuous and oriented. For $d=2$ and $(d_1,d_2) = (2,0)$, this model is solvable: it is a continuous Poissonization version of the discrete Ulam's problem described above, introduced by Hammersley \cite{Hammersley}, solved first by Logan and Shepp and by Vershik and Kerov in 1977, and revisited later in a probabilistic way by Aldous and Diaconis \cite{AldousDiaconis} using the so-called Hammersley's line process. The equality $C(2,0) = 2$ comes from there.

\paragraph{Strategy of the proof and organization of the paper.}
The common strategy of the proof of Theorems \ref{t2} and \ref{t} is the following. First prove that the first-order behavior of the time constant for small $\epsilon$ is the same for the studied model and a (partially) oriented version of it. Then we prove the convergence of the time constant of this oriented model, properly rescaled by a power of $\epsilon$, towards the time constant associated with a related semi-continuous oriented model, and check that this limit is well behaved.

The relation between the semi-continuous oriented model, the oriented model and the original one is significantly more intuitive in the context of site first-passage percolation. For this reason, we focus first on the proof of Theorem \ref{t} in Section \ref{s:proof}, following the strategy described above. The adaptation of the proof to get Theorem \ref{t2} is given in Section \ref{s:bond}. The proof of some standard results is postponed to the Appendix.

\section{Proof of Theorem \ref{t}: the site case}

\label{s:proof}

\subsection{Setting and notations}

\label{s:setting}

In the whole of Section \ref{s:proof}, we fix $d_1 \ge 1$, $d_2 \ge 0$ and we set $d=d_1+d_2$. We consider Bernoulli site first-passage percolation on $\mathbb{Z}^d$.
Our aim is to prove the existence of a positive constant $C(d_1,d_2)$ such that, for any $z \in \R^d$ such that $d_1(z)=d_1$ and $d_2(z)=d_2$,
the following limit holds: 
\[
\lim_{\epsilon \to 0} \frac{\|z\|_1-\mu_\epsilon(z)}{\epsilon^{1/d_1}} = C(d_1,d_2) \gamma(z).
\]
By symmetry, it is sufficient to prove this result for any $z \in (0,+\infty)^{d_1} \times \{0\}^{d_2}$.

We denote by  $\projh : \Z^{d_1}\times\Z^{d_2} \to \Z^{d_1}$ the projection on the first space
and by $\projv : \Z^{d_1}\times\Z^{d_2} \to \Z^{d_2}$ the projection on the second space.
We will sometimes refer to the $d_1$ first coordinates as the {\em horizontal} coordinates
and to the $d_2$ last coordinates as the {\em vertical} coordinates.
With this terminology $\projh$ is the projection on the horizontal coordinates and $\projv$ is the projection on the vertical coordinates.

In what follows, we have to deal with sequences of points in $\mathbb{R}^d$ and to adopt a notation for the coordinates of these points: we put the label of the points within the sequence into brackets, and designate one of its coordinates by a subscript. For instance, if $(w(1), \dots , w(k))$ is a sequence of $k$ points in $\mathbb{R}^d$, $w(i)_j$ is the $j$-th coordinate of $w(i)$.

\subsection{A related discrete oriented model}

\label{s:discrete_model}

\subsubsection{The model}

We first define a relation $\prec$ on $\Z^d$.
Let $x, y \in \Z^d$. 
Write $x=(x_1,\dots,x_d)$ and $y=(y_1,\dots,y_d)$.
We define $x \prec y$ as follows:
\[
x \prec y \text{ holds when, for all }  i \in \{1,\dots,d_1\}, \; x_i < y_i.
\]
In other words, we require a strict inequality for indices $i \le d_1$ and we make no requirement for indices $i \ge d_1+1$.
We define a relation $\preceq$ similarly with a large inequality.

Let $x \preceq y$ be two vertices of $\Z^{d_1} \times \{0\}^{d_2}$.
Let $\cD_\epsilon(x,y)$ ($\cD$ stands for {\em discrete}) be the set of monotone sequences of open  sites between $x$ and $y$, that is
\begin{equation*}
\cD_\epsilon(x,y): = 
 \{ (w(1),\dots,w(k)) : k \ge 0, w(1),\dots, w(k)  \text{ open sites of } \Z^d 
\text{ with } x \preceq w(1) \prec \cdots \prec w(k) \prec y\}.
\end{equation*}
Note that we commit an abuse of language by using the word "monotone" as $\prec$ is not an order.
Let $(w(1),\dots,w(k))  \in \cD_\epsilon(x,y)$. 
Write $w(0)=x$, $w(k+1)=y$ and
set
\[
V(w(1),\dots,w(k)) = \sum_{i=1}^{k+1} \|\projv(w(i))-\projv(w(i-1))\|_1 \quad \text{ and }\quad R(w(1),\dots,w(k))=k.
\]
Note that $\projv(w(0))=\projv(w(k+1))=0$ so $V$ does only depend on $(w(1),\dots,w(k))$.
Moreover
\begin{equation}\label{e:length_V}
\sum_{i=1}^{k+1} \|w(i)-w(i-1)\|_1 = \|y-x\|_1 + V(w(1),\dots,w(k)).
\end{equation}
The quantity $V(w(1),\dots,w(k))$ denotes the total vertical displacement to travel from $x$ to $w(1)$, then to $w(2)$, and so on until $w(k)$ 
and finally to $y$.
The quantity $R(w(1),\dots,w(k))$ is the number of rewards collected along this sequence, since each open vertex can be seen as a gain of one unit for the travel time.

For all $\epsilon>0$ and all $x \preceq y$ in $\Z^{d_1} \times \{0\}^{d_2}$, we define a score $S^\cD_\epsilon(x,y)$ by
\begin{equation}\label{e:score}
S^\cD_\epsilon(x,y) =  \sup_{s \in \cD_\epsilon(x,y)} \big(R(s) - V(s)\big).
\end{equation}
When the starting point $x$ is the origin of $\Z^d$, we will simply write $S^\cD_\epsilon(y)$ to denote $S^\cD_\epsilon(0,y)$. Note that
the score is non-negative as $\emptyset \in \cD_\epsilon(x,y)$ (this is the case where $k=0$) and $R(\emptyset)-V(\emptyset)=0$ (see an example of a score calculated in Figure \ref{Fig:discret_site}).
Let state here the following elementary result which quickly explains the link between this score and the travel time.
The point is that the bound given below is sharp when $\epsilon>0$ is small (see Lemma \ref{l:reduction_oriented}).
Therefore the study of $T_\epsilon$ can be reduced to the study of $S^\cD_\epsilon$.

\begin{figure}
\begin{center}
\includegraphics[width=10cm]{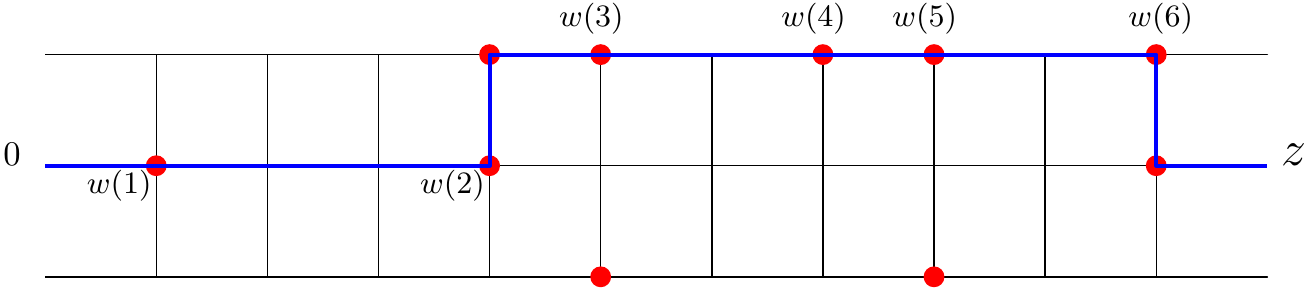}
\end{center}
\caption{Example in the discrete oriented model of site percolation in dimensions $(d_1,d_2)=(1,1)$. Open sites are drawn in red. Here, the score $S^\cD_\epsilon(0,z)$ can be achieved with the path drawn in blue. Note that two open sites $w(i),w(i+1)$ taken by the blue path are counted as two rewards for this path  only if their abscissa are strictly increasing. Note  also that the set of paths considered to compute the score can only do steps to the North, South or East. 
Here, we have $S^\cD_\epsilon(0,z)=6-2=4$. } 
\label{Fig:discret_site}
\end{figure}

\begin{lemma} \label{l:elementary}
For all $z\in \Z^{d_1}\times\{0\}^{d_2}, T_\epsilon(0,z) \le \|z\|_1 - S^\cD_\epsilon(z)$.
\end{lemma}

This lemma is clear if you have in mind that in fact $\|z\|_1 - S^\cD_\epsilon(z)$ can be seen as the travel time from $0$ to $z$ if you only allow oriented paths in the $d_1$ \emph{horizontal} direction and moreover, the passage time of an open site is counted as $0$  if and only if the previous open site through which the path passed was strictly smaller for the order $\prec$. However, we provide here also a more formal proof.
\begin{proof}
Let $s = (w(1),\dots,w(k)) \in \cD_\epsilon(0,z)$.
Write $w(0)=0$ and $w(k+1)=z$.
Consider a path $\pi$ of minimal length (number of steps) that starts from $w(0)$, goes to $w(1)$, then to $w(2)$ and so on until $w(k+1)$.
Its length is 
\[
\sum_{i=1}^{k+1} \|w(i)-w(i-1)\|_1 = \|z\|_1 + V(s) \text{ by } \eqref{e:length_V}.
\]
As each site $w(1),\dots,w(k)$ is open, its travel time can be bounded from above by its length minus $k$:
\[
\tau_\epsilon(\pi) \le \|z\|_1 + V(s) - R(s) = \|z\|_1 - (R(s) - V(s)).
\]
This yields the result.
\end{proof}

We define a mean directional score $\sigma^\cD_\epsilon(\cdot)$ in the following lemma.
\begin{lemma}  \label{l:sigma_discrete}
For all $z \in (0,+\infty)^{d_1} \times \{0\}^{d_2}$, the following limit is well defined :
\begin{equation}\label{e:sigma}
\sigma^\cD_\epsilon(z) : = \lim_{n \to \infty} \frac 1 n \E[S^\cD_\epsilon(\lfloor nz \rfloor )].
\end{equation}
\end{lemma}
%
\begin{proof}
This follows from standard sub-additive arguments.
Let $x \preceq y \preceq z$ in $\Z^{d_1}\times\{0\}^{d_2}$.
The concatenation $(s,s')$ of a sequence $s\in \cD_\epsilon(x,y)$ and of a sequence $s' \in \cD_\epsilon(y,z)$ belongs to $\cD_\epsilon(x,z)$.
Moreover $R(s,s')=R(s)+R(s')$ and $V(s,s') \le V(s)+V(s')$. Therefore 
\[
S^\cD_\epsilon(x,z) \ge S^\cD_\epsilon(x,y)+S^\cD_\epsilon(y,z).
\]
Let now $z$ be in $(0,+\infty)^{d_1} \times \{0\}^{d_2}$.
For any integers $p,q \ge 0$, using first the above triangle inequality and then the non-negativity of the scores,
\begin{align*}
S^\cD_\epsilon(\lfloor (p+q)z \rfloor) 
& \ge S^\cD_\epsilon(\lfloor pz \rfloor) + S^\cD_\epsilon(\lfloor pz \rfloor,\lfloor pz \rfloor + \lfloor qz \rfloor) +
 S^\cD_\epsilon(\lfloor pz \rfloor + \lfloor qz \rfloor,\lfloor (p+q)z \rfloor)\\
& \ge S^\cD_\epsilon(\lfloor pz \rfloor) + S^\cD_\epsilon(\lfloor pz \rfloor,\lfloor pz \rfloor + \lfloor qz \rfloor).
\end{align*}
Integrating and using stationarity, we get
\[
\E\big[S^\cD_\epsilon(\lfloor (p+q)z \rfloor)\big]
\ge \E[S^\cD_\epsilon(\lfloor pz \rfloor)\big]+\E[S^\cD_\epsilon(\lfloor qz \rfloor)\big].
\]
The result then follows from Fekete's subadditive Lemma.
\end{proof}

It is worth noticing that $\|z\|_1 - \sigma_\epsilon^\cD(z)$ corresponds to the time constant associated with the discrete oriented model we have just defined. The convergence appearing in Lemma \ref{l:sigma_discrete} could be strengthened to a convergence with probability one and in $L^1$, however we do not need it.

\subsubsection{Link between the oriented model and the original model}

\label{s:link_oriented_original}

We have in hand two models: the site first-passage percolation, and its (partially) oriented version. We defined the two corresponding time constants, namely $\mu_\epsilon(z)$ and $\|z\|_1-\sigma_\epsilon^\cD(z)$. Our goal is now to prove that these two time constants have the same behavior at order $\epsilon^{1/d_1}$. More precisely, the main result of this section is the following result.

\begin{prop} \label{l:reduction_oriented} 
For all $z\in(0,+\infty)^{d_1}\times\{0\}^{d_2}$,
\[
\lim_{\epsilon \to 0} \frac{\mu_\epsilon(z)-(\|z\|_1-\sigma_\epsilon^\cD(z))}{\epsilon^{1/d_1}} = 0.
\]
\end{prop}

Lemmas \ref{l:elementary} and \ref{l:sigma_discrete} already tell us that $\mu_\epsilon(z) \leq \|z\|_1-\sigma_\epsilon^\cD(z) $. This comes from a basic comparison between general first-passage percolation and (partially) oriented first-passage percolation: restricting ourselves to oriented paths increases the minimal passage time over paths. The delicate part is to prove that it cannot increases it {\em significantly}.
In the remaining of Section \ref{s:link_oriented_original} we give a proof relying 
on Lemma \ref{l:path-nice} -- whose standard proof is given in Appendix \ref{s:proof-path-nice} --
and on Lemma \ref{l:sommable} -- whose technical proof is postponed to Section  \ref{s:proof-sommable}. 

We first need to introduce a few objects and two intermediate results.
Recall the definition of $\preceq$ in Section \ref{s:discrete_model}. 
Let $x, y \in \Z^{d_1}\times\{0\}^{d_2}$.
We denote by $y^-$ the point $y-\1_{d_1}$ where $\1_{d_1}$ is the point of $\Z^{d_1} \times \{0\}^{d_2}$ whose $d_1$ first coordinates equal one.
Assume $x \prec y^-$.
We say that a path $\pi=(a(0),\dots,a(n))$ in $\Z^d$ is a nice path from $x$ to $y^-$ if 
\[
a(0)=x, a(n)=y^- \text{ and for all } i \in \{0,\dots,n\}, a(i) \prec y.
\]
We denote by $T_\epsilon^{\text{nice}}(x,y)$ the infimum of the travel times $\tau_\epsilon(\pi)$ over every nice path $\pi$ from $x$ to $y^-$.

\begin{lemma} \label{l:path-nice}
For all $z \in (0,+\infty)^{d_1}\times \Z^{d_2}$, for all $\epsilon>0$,
\[
\mu_\epsilon(z) = \lim_{n \to \infty} \frac 1 n \E\big[T^{\text{nice}}_\epsilon(0,\lfloor nz \rfloor)\big].
\]
\end{lemma}

This kind of results belongs to the folklore of first-passage percolation.
We provide the short proof in Section \ref{s:proof-path-nice}. 
For all $z \in (0,+\infty)^{d_1} \times \{0\}^{d_2}$, $\epsilon>0$, $\eta>0$ and $n$ we consider the event
\[
\cM(z,\epsilon,\eta,n)  = \{T^{\text{nice}}_\epsilon(0, \lfloor nz \rfloor) < \|\lfloor nz \rfloor^-\|_1 
- S^\cD_\epsilon(\lfloor nz \rfloor) - \eta n \epsilon^{1/d_1}\}.
\]

\begin{lemma}\label{l:sommable} 
For all $z \in (0,+\infty)^{d_1} \times \{0\}^{d_2}, \eta>0$ and for all $\epsilon>0$ small enough (depending on $z$ and $\eta$),
\[
\lim_{n\to \infty} \P[\cM(z,\epsilon,\eta,n)] = 0.
\]
\end{lemma}

The basic intuition is quite simple. 
We are interested in geodesics from $0$ to a point $x$ in $(\N^*)^{d_1} \times \{0\}^{d_2}$.
The travel time of a path is its length minus the number of open sites it visits.
Using steps in $-e_1, \dots, -e_{d_1}$ makes the path longer.
When $\epsilon$ decreases to $0$, making such steps to reach an open site becomes too costly.
In the limit when $\epsilon$ tends to $0$, 
we can thus restrict our attention to paths which use no such steps 
and we can actually restrict to paths that only collects open sites $b(1), \dots, b(p)$ such that $(b(1),\dots,b(p))$ belongs to $\cD_\epsilon(x)$.
In other words, in the regime we are interested in, we can replace $T_\epsilon(0,x)$ by $\|x\|_1 - \cS^\cD_\epsilon(0,x)$.
The proof of this result is actually rather technical. 
We give it in Section \ref{s:proof-sommable}.

\begin{proof}[Proof of Proposition \ref{l:reduction_oriented} using Lemmas \ref{l:path-nice} and \ref{l:sommable}]
Fix $z \in (0,+\infty)^{d_1} \times \{0\}^{d_2}$ and $\eta>0$.
By Lemma \ref{l:sommable} we can fix $\epsilon_0=\epsilon_0(z,\eta)>0$ such that, for all $\epsilon \in (0,\epsilon_0)$,
\[
\lim_{n\to \infty} \P[\cM(z,\epsilon,\eta,n)] = 0.
\]
By Lemma \ref{l:elementary}, for all $n \ge 1$,
\[
T_\epsilon(0,\lfloor nz \rfloor) - n\|z\|_1 + S^\cD_\epsilon(\lfloor nz \rfloor) \le 0
\]
and therefore by the definition of the time constant $\mu_\epsilon (z)$ (see Lemma \ref{l:cte} in Appendix) and Lemma \ref{l:sigma_discrete},
\[
\mu_\epsilon(z) - \|z\|_1 + \sigma^\cD_\epsilon(z) \le 0.
\]
Moreover, for all $n \ge 1$,
\[
T^{\text{nice}}_\epsilon(0,\lfloor nz \rfloor) \ge 0 \text{ and } S^\cD_\epsilon(\lfloor nz \rfloor) \ge 0
\]
and then
\[
T^{\text{nice}}_\epsilon(0,\lfloor nz \rfloor) - \|\lfloor nz\rfloor^-\|_1 + S^\cD_\epsilon(\lfloor nz \rfloor)
\ge 
 -\eta n \epsilon^{1/d_1}\1_{\cM(z,\epsilon,\eta,n)^c} 
-
 \|\lfloor nz\rfloor^-\|_1\1_{\cM(z,\epsilon, \eta,n)}  .
\]
Therefore, using Lemmas \ref{l:sigma_discrete}, \ref{l:path-nice} and \ref{l:sommable}, we get
\[
\mu_\epsilon(z) - \|z\|_1 + \sigma^\cD_\epsilon(z) \ge -\eta \epsilon^{1/d_1}.
\]
This ends the proof of the proposition.
\end{proof}

\subsection{A related semi-continuous oriented model}
\label{s:continuous_model}
\subsubsection{The model}

We define an oriented semi-continuous model on $\R^{d_1} \times \Z^{d_2}$ as follows.
Let $\d x$ denote the Lebesgue measure on $\R^{d_1}$ and $\d v$ denote the counting measure on $\Z^{d_2}$.
Let $\xi$ be a Poisson point process on $\R^{d_1} \times \Z^{d_2}$ with intensity $\nu=\d x \otimes \d v$. 
In other words,
\[
\xi = \bigcup_{v \in \Z^{d_2}} \{(x,v), x \in \chi_v\}
\]
where $(\chi_v)_{v \in \Z^{d_2}}$ is a family of independent Poisson point processes on $\R^{d_1}$ with intensity $\d x$.
We call {\em particles} the points of $\xi$.

Let $\epsilon>0$.
For $a=(a_1,\ldots,a_{d_1}) \in \Z^{d_1}$ and $b \in \Z^{d_2}$, we define the $\epsilon-$cube $C_\epsilon(a,b)$ by
\begin{equation}\label{eq:def_Cube}
C_\epsilon(a,b) = \prod_{i=1}^{d_1}[a_i\epsilon^{1/d_1}, (a_i+1)\epsilon^{1/d_1}) \times \{b\}.
\end{equation}

We say that a $\epsilon$-cube is open if it contains at least one particle of $\xi$.
Otherwise we say that the cube is closed.
As $\nu[C_\epsilon(a,b)] = \epsilon$, each cube is open with probability $\tilde\epsilon = 1-\exp(-\epsilon)$.
Set 
\[
\tau_{\tilde\epsilon}(a,b) = \1_{C_\epsilon(a,b) \text{ is closed}}.
\]
The family $(\tau_{\tilde\epsilon}(a,b))_{a,b}$ is a family of independent Bernoulli random variables with parameter $1-\tilde\epsilon$.
We use this family to define the discrete oriented model with  parameter $\tilde\epsilon$.
Under this coupling, $(a,b) \in \Z^d$ is open (in the discrete model) if and only if $C_\epsilon(a,b)$ is open (in the semi-continuous model).
We refer to Figure \ref{Fig:discret_continu}.

\begin{figure}
\begin{center}
\includegraphics[width=12cm]{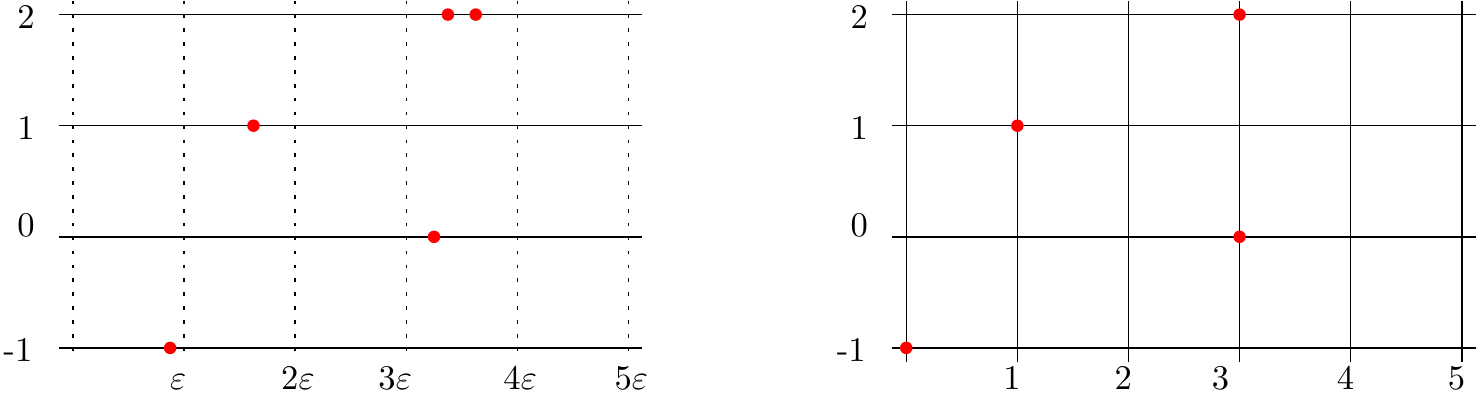}
\end{center}
\caption{Correspondence between the semi-continuous model and the discrete model in dimensions $(d_1,d_2)=(1,1)$. On the left, red points are distributed on each horizontal line as a Poisson point process with unit intensity. On the right, the  corresponding discrete site percolation on $\Z^2$ with parameter $\exp(-\varepsilon)$ (red sites have time $0$ ; a site is red with probability $\tilde\epsilon=1-\exp(-\epsilon)$).
Note that the horizontal scale is not the same on the left and on the right.} 
\label{Fig:discret_continu}
\end{figure}

We extend the definition of $\prec$ and $\preceq$ from $\Z^{d_1}\times \Z^{d_2}$ to $\R^{d_1} \times \Z^{d_2}$ in a natural way.
For any $\alpha \ge 0$, we introduce a new relation $\prec_\alpha$ on $\R^{d_1} \times \Z^{d_2}$ by
\begin{equation}
\label{e:ajoutM3}
x \prec_\alpha y \text{ holds when, for all }  i \in \{1,\dots,d_1\}, \; x_i + \alpha^{1/d_1} < y_i.
\end{equation}
Note that $\prec_0$ is equal to $\prec$.
For any $z \in (0,+\infty)^{d_1}\times \{0\}^{d_2}$ we denote by
$\cC_\alpha(z)$ ($\cC$ stands for {\em continuous}) the set of all monotone sequences of particles of $\xi$ between $0$ and $z$ which are $\alpha$-separated, which we define as the set
\[
\cC_\alpha(z)=
\{ (w(1),\dots,w(k)) : k \ge 0, w(1),\dots, w(k)  \in \xi
\text{ such that } 0 \preceq w(1) \prec_\alpha \cdots \prec_\alpha w(k) \prec_\alpha z\}.
\]
We extend in a natural way the definition of $R$ and $V$ from sequences of $\Z^{d_1}\times \Z^{d_2}$ to sequences of $\R^{d_1} \times \Z^{d_2}$.
We then define a score $S^{\cC}_\alpha(z)$ by
\[
S^{\cC}_\alpha(z)= \sup_{s \in \cC_\alpha(z)}\big(R(s)-V(s)).
\]
As in the discrete setting, we define a mean directional score $\sigma^\cC_\alpha$ as follows.

\begin{lemma} \label{l:sigmac} For all $z \in (0,+\infty)^{d_1} \times \{0\}^{d_2}$ and any $\alpha \ge 0$,
\[
\sigma^\cC_\alpha(z) 
:= \lim_{\lambda \to \infty \,,\, \lambda \in \mathbb{R}} \frac1 \lambda \E\big[S^{\cC}_\alpha(\lambda z)\big] 
= \sup_{\lambda \in \mathbb{R}^{+*}} \frac1 \lambda \E\big[S^{\cC}_\alpha(\lambda z)\big]\in (0,\infty].
\]
\end{lemma}

\begin{proof}
As in the proof of Lemma \ref{l:sigma_discrete} one checks that for any $\lambda, \lambda' \ge 0$,
\[
\E\big[S^{\cC}_\alpha((\lambda +\lambda') z)\big] \ge \E\big[S^{\cC}_\alpha(\lambda z)\big]+\E\big[S^{\cC}_\alpha(\lambda' z)\big].
\]
As $\E\big[S^{\cC}_\alpha(\lambda z)\big]$ is moreover non-negative for any $\lambda$, one deduces from a continuous version of
Fekete's subadditive lemma the required result.\footnote{Define $f:(0,+\infty)\to\R$ by $f(\lambda)=\E\big[S^{\cC}_\alpha(\lambda z)\big]$.
We know that $f$ is super-additive and non-negative.
Fix $\lambda_0 > 0$.
For any $\lambda>0$ write $\lambda = q(\lambda) \lambda_0 + r(\lambda)$ with $q(\lambda)\in \N$ and $r(\lambda) \in [0,\lambda_0)$.
As $f$ is super-additive and non-negative we get $f(\lambda) \ge q(\lambda) f(\lambda_0)$ and thus
\[
\liminf_{\lambda \to +\infty} \frac 1 \lambda f(\lambda) \ge  \liminf_{\lambda \to +\infty} \frac{q(\lambda)}{\lambda} f(\lambda_0) 
= \frac 1 {\lambda_0} f(\lambda_0).
\]
Therefore 
\[
\liminf_{\lambda \to +\infty} \frac 1 \lambda f(\lambda) \ge \sup_{\lambda_0>0} \frac 1 {\lambda_0} f(\lambda_0).
\]
As the inequality $\limsup_{\lambda \to +\infty} \frac 1 \lambda f(\lambda) \le \sup_{\lambda_0>0} \frac 1 {\lambda_0} f(\lambda_0)$
is straightforward, this ends the proof.}
\end{proof}

\subsubsection{Link between the oriented discrete model and the oriented semi-continuous model}\label{s:link_discret_continu}

We have already proved (see Proposition \ref{l:reduction_oriented}) that the time constant $\mu_\epsilon (z)$ behaves like $\|z\|_1 -\sigma_\epsilon^\cD(z) + o (\epsilon^{1/d_1})$ for small $\epsilon$. We now want to prove that $\sigma_\epsilon^\cD(z)$ is indeed of order $\epsilon^{1/d_1}$. This is done by proving that $\sigma_\epsilon^\cD(z)/ \epsilon^{1/d_1}$ actually converges to the mean continuous directional score $\sigma_0^{\cC}(z)$.
The aim of this section is thus to prove the following result.

\begin{lemma}\label{l:link_discrete_continous} For all $z \in (0,+\infty)^{d_1} \times \{0\}^{d_2}$,
\[
\lim_{\tilde\epsilon \to 0} \frac{\sigma^\cD_{\tilde\epsilon} (z)}{\tilde\epsilon^{1/{d_1}}} = \sigma_0^{\cC}(z).
\]
\end{lemma}

To achieve this goal, we need two ingredients. The first one is a study of the dependence in $\alpha$ of $S^\cC_{\alpha}$ (see Lemma \ref{l:sigmac2}), that allows us to approximate $ \sigma_0^{\cC}$ by $\sigma_\alpha^{\cC}$ for $\alpha$ small enough. The second one (see Lemma \ref{l:coupling1}) is a comparison between the discrete score $S^\cD_{\tilde\epsilon}$ and two of its continuous counterparts, namely $S^\cC_{0}$ and $S^\cC_{\epsilon}$ for some $\epsilon$ depending on $\tilde\epsilon$.

We first prove the two following lemmas.

\begin{lemma}\label{l:sigmac2} Let $z \in (0,+\infty)^{d_1} \times \{0\}^{d_2}$.
\begin{enumerate}
\item For all $\alpha \ge 0$, the maps $\alpha \mapsto \E\big[S^{\cC}_\alpha(z)\big]$ and $\alpha \mapsto \sigma^{\cC}_\alpha(z)$ are non-increasing.
\item $\displaystyle{\lim_{\alpha \to 0} \E\big[S^{\cC}_\alpha(z)\big] = \E\big[S^{\cC}_0(z)\big]}$.
\item $\displaystyle{\lim_{\alpha \to 0} \sigma^{\cC}_\alpha(z) = \sigma^{\cC}_0(z)}$.
\end{enumerate}
\end{lemma}

Notice that the quantity $S^{\cC}_\alpha(z)$ is designed to be monotone in $\alpha$, which makes straightforward the convergence when $\alpha$ goes to $0$.

\begin{proof} The monotony of $\alpha \mapsto \E\big[S^{\cC}_\alpha(z)\big]$ 
is a consequence of the monotony (in the sense of inclusion) of $\alpha \mapsto \cC_\alpha(z)$.
The monotony of $\alpha \mapsto \sigma^{\cC}_\alpha(z)$ follows.
Then,
\begin{align*}
\lim_{\alpha \to 0} \E\big[S^{\cC}_\alpha(z)\big]
& = \E\left[\lim_{\alpha \to 0} S^{\cC}_\alpha(z)\right], \text{ by the monotone convergence theorem} \\
& = \E\left[\sup_{\alpha > 0} S^{\cC}_\alpha(z)\right], \text{ by monotonicity} \\
& = \E\left[\sup_{\alpha > 0} \sup_{s \in \cC_\alpha(z)}\big(R(s)-V(s))\right], \\
& = \E\left[\sup_{s \in \cC_0(z)}\big(R(s)-V(s))\right], \text{ as } \bigcup_{\alpha>0} \cC_\alpha(z)= \cC_0(z).
\end{align*}
This gives the second item of the lemma.
Then,
\begin{align*}
\lim_{\alpha \to 0} \sigma^{\cC}_\alpha(z)
 & = \sup_{\alpha > 0 } \sigma^{\cC}_\alpha(z), \text{ by monotonicity}\\
 & = \sup_{\alpha > 0 } \sup_{\lambda >0} \frac1 \lambda \E\big[S^{\cC}_\alpha(\lambda z)\big], \text{ by Lemma } \ref{l:sigmac} \\
 & = \sup_{\lambda >0} \frac1 \lambda \sup_{\alpha > 0 } \E\big[S^{\cC}_\alpha(\lambda z)\big],  \\
 & = \sup_{\lambda >0} \frac1 \lambda \E\big[S^{\cC}_0(\lambda z)\big], \text{ by the second item and monotonicity} \\
 & = \sigma^{\cC}_0(z),\text{ by Lemma } \ref{l:sigmac}.
\end{align*}
This ends the proof.
\end{proof}

\begin{lemma} \label{l:coupling1}
For all $z \in (0,+\infty)^{d_1}\times\{0\}^{d_2}$,
\[
\E\big[S^\cC_\epsilon(\epsilon^{1/d_1} z)\big] \le \E\big[S^\cD_{\tilde\epsilon}(\lfloor z \rfloor)\big] \le \E\big[S^\cC_0(\epsilon^{1/d_1} z)\big]
\]
where $\tilde{\varepsilon}:=1-\exp(-\varepsilon)$.
\end{lemma}

\begin{proof}
Let $z \in (0,+\infty)^{d_1}\times\{0\}^{d_2}$. 
Let $s=(w(1),\dots,w(k)) \in \cC_\epsilon(\epsilon^{1/d_1} z)$. 
Each $w(i)$ belongs to a unique cube $C_\epsilon(w'(i))$. By construction, $w'(i)$ is an open site in the discrete model with parameter $1-\tilde{\varepsilon}$.
Hence, this defines a sequence of $\tilde{\varepsilon}$-open sites $s'=(w'(1),\dots,w'(k))$. Besides, by definition, the $\epsilon$-separation of the sequence $s$ implies that the discrete sequence $s'$ is monotone for the relation $\prec$ and so $s' \in \cD_{\tilde\epsilon}(0,\lfloor  z \rfloor)$. Moreover, we clearly have 
\[
 R(s)=R(s') \text{ and } V(s)=V(s').
\]
Thus 
\[
S^\cC_\epsilon(\epsilon^{1/d_1} z) \le S^\cD_{\tilde\epsilon}(\lfloor z \rfloor)
\]
and then
\[
\E\big[S^\cC_\epsilon(\epsilon^{1/d_1} z)\big] \le \E\big[S^\cD_{\tilde\epsilon}(\lfloor  z \rfloor)\big].
\]
Now let $s=(w(1),\dots,w(k)) \in \cD_{\tilde\epsilon}(0,\lfloor  z \rfloor)$.
Each cube $C_\epsilon(w(i))$ is open.
Therefore we can take in each $C_\epsilon(w(i))$ a particle $w'(i)$. The strict monotony of the horizontal coordinates of the sequence $(w(1),\ldots,w(k))$  implies that the sequence 
 $s'=(w'(1),\dots,w'(k))$  belongs to $\cC_0 (\epsilon^{1/d_1} z)$.
As above, $R(s)=R(s')$ and $V(s)=V(s')$ and we get
\[
\E\big[S^\cD_{\tilde\epsilon}(\lfloor z \rfloor)\big] \le \E\big[S^\cC_0(\epsilon^{1/d_1}  z)\big]. 
\]
This ends the proof.
\end{proof}

\begin{proof}[Proof of Lemma \ref{l:link_discrete_continous}]
From Lemma \ref{l:coupling1} we get
\[
\lim_{n\to \infty} \frac 1 n \E\big[S^\cC_\epsilon( n \epsilon^{1/d_1} z )\big] 
\le 
\lim_{n\to \infty} \frac 1 n \E\big[S^\cD_{\tilde\epsilon}( \lfloor n  z \rfloor)\big] 
\le 
\lim_{n\to \infty} \frac 1 n \E\big[S^\cC_0( n \epsilon^{1/d_1} z )\big].
\]	
From the definitions of $\sigma^\cD(\cdot)$ and $\sigma^\cC_\cdot(\cdot)$ we deduce
\[
\epsilon^{1/d_1} \sigma_\epsilon^{\cC}(z) \le \sigma^\cD_{\tilde\epsilon}(z) \le \epsilon^{1/d_1} \sigma_0^{\cC}(z).
\]
By the third item of Lemma \ref{l:sigmac2} we then get
\[
\lim_{\epsilon \to 0} \frac{\sigma^\cD_{\tilde\epsilon} (z)}{\epsilon^{1/{d_1}}} = \sigma_0^{\cC}(z).
\]
The result follows as $\tilde\epsilon = 1-\exp(-\epsilon) \sim \epsilon$ when $\epsilon \to 0$. 
\end{proof}

\subsubsection{Study of $\sigma^{\cC}_0$}

We proved in Lemma \ref{l:link_discrete_continous} that $\sigma_\epsilon^\cD(z)/ \epsilon^{1/d_1}$ converges to the mean continuous directional score $\sigma_0^{\cC}(z)$. It remains to check that this limit is finite, and to clarify its dependence in $z$.

Recall that $\gamma(z)$ is the geometric mean of the non-zero coefficients of $z$ (see \eqref{e:gamma}).
Recall also that $\1_{d_1}$ denotes the vector of $(0,+\infty)^{d_1} \times \{0\}^{d_2}$ whose non-zero coefficients equal $1$.

\begin{lemma} \label{l:valuesigmac}
For all $z \in (0,+\infty)^{d_1} \times \{0\}^{d_2}$, 
\[
\sigma^\cC_0(z) = \gamma(z) \sigma^\cC_0(\1_{d_1}) < \infty.
\]
\end{lemma}

The proof is divided in two steps. We first prove that $\sigma^\cC_0(z) = \gamma(z) \sigma^\cC_0(\1_{d_1})$. This follows easily from a scaling argument for Poisson point processes. Then we prove that $ \sigma^\cC_0(\1_{d_1}) < \infty$. This is done through a control of the tail of the distribution of $\cS^\cC_0(\lambda \1_{d_1})$.

\begin{proof}
Let $z \in (0,+\infty)^{d_1} \times \{0\}^{d_2}$ and write $\gamma$ for $\gamma(z)$.
Let $\phi : \R^d \to \R^d$ be the linear map defined by 
\[
\phi(x_1,\dots,x_d) = (\gamma x_1/z_1, \dots, \gamma x_{d_1}/z_{d_1}, x_{d_1+1},\dots,x_d).
\]
Then $\phi(z)=\gamma \1_{d_1}$.
Making explicit in the notations the dependence on the configuration $\xi$ of particles, we have
\begin{align*}
\sup_{s \in \cC_0(z)(\xi)}\big(R(s)-V(s)\big) 
& = \sup_{s \in \cC_0( \gamma \1_{d_1})(\phi(\xi))}\big(R(\phi^{-1}(s))-V(\phi^{-1}(s))\big) \\
& = \sup_{s \in \cC_0( \gamma \1_{d_1})(\phi(\xi))}\big(R(s)-V(s)\big).
\end{align*}
But $\phi$ preserves the Lebesgue measure and therefore $\phi(\xi)$ has the same distribution as $\xi$.
Integrating the previous equality yields
\[
\E\big[S^{\cC}_0(z)\big] =  \E\big[S^{\cC}_0(\gamma(z)\1_{d_1})\big].
\] 
Applying this for all $\lambda z$ we get
\[
\lim_{\lambda \to \infty } \frac 1 \lambda \E\big[S^{\cC}_0(\lambda z)\big] 
= 
\lim_{\lambda \to \infty } \frac 1 \lambda \E\big[S^{\cC}_0(\lambda \gamma(z)\1_{d_1})\big] 
= \gamma(z) \lim_{\lambda \to \infty } \frac 1 \lambda \E\big[S^{\cC}_0(\lambda\1_{d_1})\big]
\]
and thus
\[
\sigma^\cC_0(z) = \gamma(z) \sigma^\cC_0(\1_{d_1}).
\]
It remains to prove that $\sigma^\cC_0(\1_{d_1})$ is finite.
Let $A>0$ and $\lambda>0$.
We have
\begin{align*}
\P[\cS^\cC_0(\lambda \1_{d_1}) \ge\lambda A] 
 & \le \E\Big[\sum_{n \ge \lambda A}\sum_{s=(w(1), \dots, w(n)) \in \cC_0(\lambda \1_{d_1})} \1_{\{R(s)-V(s) \ge \lambda A\}} \Big]\\
 & =  \sum_{n \ge\lambda A} \E\Big[\sum_{w(1), \dots, w(n) \in \xi} \1_{\{\forall j \in \{1,\dots,d_1\},  0 \le w(1)_j < \cdots < w(n)_j < \lambda\}}
 \1_{\{n-V(w(1), \dots, w(n)) \ge \lambda A\}} \Big].
\end{align*}
Using the multivariate Mecke formula (see Theorem 4.4 in \cite{last_penrose_2017}) for the point process $\xi$, we get 
\begin{align*}
& \P[\cS^\cC_0(\lambda \1_{d_1}) \ge \lambda A]  \\
 & \le \sum_{n \ge \lambda A} \int_{(\R^{d_1})^n} \d a(1) \dots \d a(n) \sum_{b(1), \dots, b(n) \in \Z^{d_2}} 
 \1_{\{\forall j \in \{1,\dots,d_1\},  0 \le a(1)_j < \cdots < a(n)_j < \lambda\}}
 \1_{\{n-V((a(1),b(1)), \dots, (a(n),b(n))) \ge \lambda A\}}.
\end{align*}
With a change of variable we get, for all $n \ge 1$,
\[
\sum_{b(1), \dots, b(n) \in \Z^{d_2}}  \1_{\{n-V((a(1),b(1)), \dots, (a(n),b(n))) \ge\lambda A\}} 
\le
\sum_{\Delta(1), \dots, \Delta(n) \in \Z^{d_2}} \1_{\{n-\sum_{i=1}^n \|\Delta(i)\|_1 \ge \lambda A\}} 
\le e^{n-\lambda A}  K^{d_2n}
\]
where
\[
K=\sum_{u \in \Z} \exp(-|u|) < \infty.
\]
On the other hand,
\[
\int_{(\R^{d_1})^n} \d a(1) \dots \d a(n)  \1_{\{\forall j \in \{1,\dots,d_1\},  0 \le a(1)_j < \cdots < a(n)_j < \lambda\}} 
= \Big(\frac{\lambda^n}{n!}\Big)^{d_1}\le \Big(\frac{e\lambda}n\Big)^{nd_1}.
\]
Thus,
\begin{align*}
\P[\cS^\cC_0(\lambda \1_{d_1}) \ge \lambda A]  
& \le \sum_{n \ge \lambda A} e^{n-\lambda A}  K^{d_2n}\Big(\frac{e\lambda}n\Big)^{nd_1}   \\
& = e^{-\lambda A} \sum_{n \ge \lambda A}\Big(\frac{e^{1+d_1}\lambda^{d_1}K^{d_2}}{n^{d_1}}\Big)^n \\
& \le e^{-\lambda A} \sum_{n \ge \lambda A}\Big(\frac{e^{1+d_1}K^{d_2}}{A^{d_1}}\Big)^n \\
& \le e^{-\lambda A} \sum_{n \ge 1}\Big(\frac{e^{1+d_1}K^{d_2}}{A^{d_1}}\Big)^n. \\
\end{align*}
Fix $A_0>0$ such that 
\[
\frac{e^{1+d_1}K^{d_2}}{A_0^{d_1}} = \frac 1 2.
\]
Then, for all $A \ge A_0$,
\[
\P[\cS^\cC_0(\lambda \1_{d_1}) \ge \lambda A]  \le e^{-\lambda A}.
\]
Therefore, 
\[
\E\left[\frac{\cS^\cC_0(\lambda \1_{d_1})}{\lambda}\right]\le A_0 + \int_{A_0}^{+\infty} e^{-\lambda A} \d A \le A_0 + \frac 1 \lambda
\]
and then
\[
\sigma^{\cC}_0(\1_{d_1}) = \lim_{\lambda \to \infty} \E\Big[\frac{\cS^\cC_0(\lambda \1_{d_1})}{\lambda}\Big]  \le A_0 < \infty.
\]
\end{proof}

\subsection{Proof of Theorem \ref{t}} 

\label{s:end_proof}

The proof of Theorem \ref{t} is now straightforward (but recall that we postponed the proof of Lemma \ref{l:sommable}). 
Write
\[
\frac{\|z\|_1-\mu_\epsilon(z)}{\epsilon^{1/d_1}} 
= \frac{\|z\|_1-\sigma_\epsilon^\cD(z)-\mu_\epsilon(z)}{\epsilon^{1/d_1}} + \frac{\sigma_\epsilon^\cD(z)}{\epsilon^{1/d_1}}.
\]
The first term tends to $0$ when $\epsilon$ tends to $0$ by Proposition \ref{l:reduction_oriented}.
The second term tends to $\sigma_0^\cC(z)$ by Lemma \ref{l:link_discrete_continous}.
By Lemma \ref{l:valuesigmac} we thus get 
\[
\lim_{\epsilon\to 0}  \frac{\|z\|_1-\mu_\epsilon(z)}{\epsilon^{1/d_1}}  =  \gamma(z) \sigma^\cC_0(\1_{d_1}).
\]
Thus the theorem holds with 
\begin{equation}
\label{e:expressionC}
C(d_1,d_2)= \sigma^\cC_0(\1_{d_1})
\end{equation}
which is finite by Lemma \ref{l:valuesigmac}.
Let us recall
 that $\sigma^\cC_0(\1_{d_1})$ is the directional score in direction $\1_{d_1}$ in the semi-continuous model of dimension $d_1+d_2$. So, $\sigma^\cC_0(\1_{d_1})$  also depends on $d_2$ although it is not recalled in our notation.

\subsection{Proof of Lemma \ref{l:sommable}}
\label{s:proof-sommable}

In this section we complete the proof of Theorem \ref{t} by giving a proof of Lemma \ref{l:sommable}. The idea is to associate with each nice path from $0$ to $\lfloor nz \rfloor^-$ an oriented path in the \emph{horizontal} directions which have almost the same travel time. To this end, we will  need first to introduce some notations.

We refer to Figures \ref{Fig:squelette}, \ref{Fig:detour} and \ref{Fig:diagonale} for an illustration of all the notations we introduce now.
Let $(b(1), \dots,b(p))$ be a sequence of sites of $\Z^d$ such that 
\begin{equation}\label{e:becausenice}
b(i) \prec\lfloor nz \rfloor \text{ for all } i \in \{1,\dots, p\}.
\end{equation}
Set also $b(0)=0$ and $b(p+1)=\lfloor nz \rfloor$.
We  extract a finite monotone sub-sequence $(b(i_k))_{0\le k\le q}$ of $(b(0),\dots,b(p+1))$ recursively  as follows: 
\begin{itemize}
\item[$\bullet$] Set $i_0=0$
\item[$\bullet$] If $i_k< p+1$,  set $i_{k+1}=\inf\{i>i_k, b(i_k) \prec b(i)\}$. 
\end{itemize}
Note that, if $i_k< p+1$, Equation \eqref{e:becausenice} implies that $\inf\{i>i_k, b(i_k) \prec b(i)\}\le p+1$ and so $i_{k+1}$ is well defined.
We thus get an integer $q\ge 1$ and a sequence $b(i_0) \prec \cdots \prec b(i_q)$ of length $q+1$ such that  $b(i_0)=b(0)=0$, $b(i_q)=b(p+1)=\lfloor nz \rfloor$.
For each $k \in \{0,\dots,q-1\}$ we write
\[
\ell_k=i_{k+1}-i_{k}-1 \ge 0 \text{ and } (c(k,0),\dots,c(k,\ell_k)) = (b(i_{k}), \dots, b(i_{k+1}-1)).
\]
In other words, 
$c(k,0)$ is a new notation for $b(i_k)$ 
and $(c(k,1),\dots,c(k,\ell_k))$ is the sequence of $b(i)$ that are strictly between (in the order given by the index) $b(i_{k})$ and $b(i_{k+1})$.
Note that $\ell_k$ can be equal to $0$.
We also write $c(q,0)=b(i_q)=\lfloor nz \rfloor$.
By construction, 
\begin{equation}\label{e:salut}
\forall k \in \{0,\dots,q-1\}, \forall i \in \{1,\dots,\ell_k\}, c(k,0) \not\prec c(k,i)
\end{equation}
and
\begin{equation}\label{e:hi}
0 = c(0,0) \prec c(1,0) \prec \dots \prec c(q,0) = \lfloor nz \rfloor.
\end{equation}

For each $i \in \{1,\dots,p\}$ write 
\[
f(i) = \projv(b(i) - b(i-1)).
\]
For all $k \in \{0,\dots,q-1\}$ and $i \in \{1,\dots,\ell_k\}$, set
\[
g(k,i)=\projh(c(k,i)-c(k,i-1)) \text{ and } h(k)=\projh(c(k+1,0)-c(k,0))=\projh(b(i_{k+1})-b(i_k)).
\]
Note also that we have the following relation between $p,q$ and the sequence $(\ell_k)_{0\le k \le q-1}$:
\[
p=q-1+\sum_{k=0}^{q-1} \ell_k.
\]
An illustration of these notations is given in Figures \ref{Fig:squelette} and \ref{Fig:detour} in dimension $d_1=d_2=1$.

Moreover, by construction of the monotone sub-sequence, for each $k\in\{0,\ldots,q-1\}$, if $\ell_k\neq 0$,
$$b(i_k)  \not\prec b(i_{k+1}-1)= c(k,\ell_k) .$$
Hence, we can define some direction  $j_k\in\{1,\ldots,d_1\}$ such that
\begin{equation}\label{e:properties_detours_1}
j_k:=\inf\{j,\; b(i_{k+1}-1)_{j}\le b(i_k)_{j}\}.
\end{equation}
In words, $j_k$ is an \emph{horizontal} direction for which $b(i_{k+1}-1)$ is \emph{on the left} of 
$b(i_k)$. We set $j_k=1$ if $\ell_k=0$. An example is drawn in Figure \ref{Fig:diagonale} in dimension $d_1=2$ and $d_2=0$.

\begin{figure}
\begin{center}
\includegraphics[width=11.8cm]{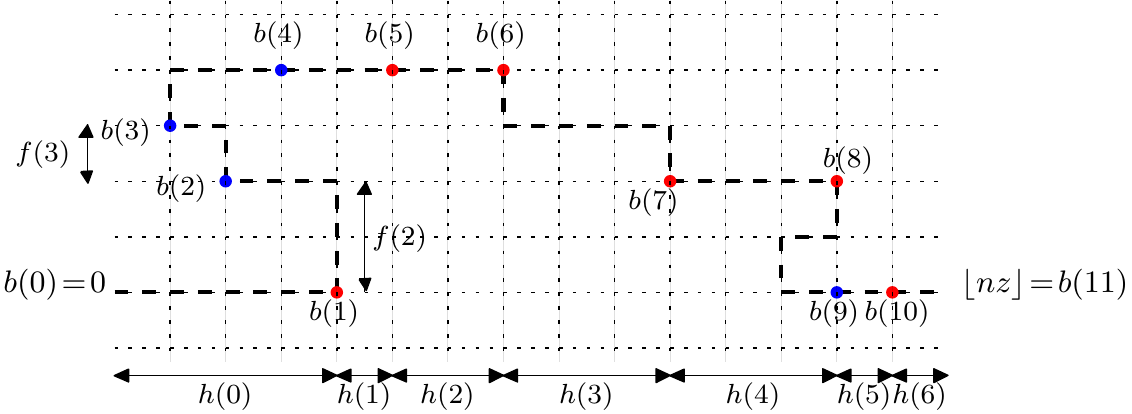}
\end{center}
 \caption{Illustration in dimensions $(d_1,d_2)=(1,1)$ of the notations $(b(i),f(i),q,\ell_k,h(k))$ introduced in Section \ref{s:proof-sommable}. The sites $(b(i))_{1\le i\le  10}$ are  given. For $1\le i \le 10$, the vector $f(i)$  of $\Z^{d_2}$ is defined as the \emph{vertical} increment between the two consecutive sites $b(i-1)$ and $b(i)$.
 The sites $b(i)$ drawn in red correspond to the monotone sub-sequence $(b(i_k))_{k\ge 1}$. The integer $q-1$ is equal to the number of red sites, so here $q=7$.  For $0\le k < q$, the vector $h(k)$  of $\Z^{d_1}$ is defined as the \emph{horizontal} increment between the two consecutive  red sites $b(i_{k})$ and $b(i_{k+1})$ and the integer $\ell_k$ denotes the number of blue sites between these two  red sites. For example, here  $\ell_1=3$, $\ell_5=1$ and $\ell_k=0$ for $k\notin \{1,5\}$.  }
\label{Fig:squelette}
\begin{center}
\includegraphics[width=3.4cm]{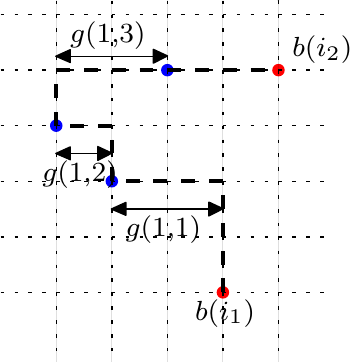}
\end{center}
       \caption{Illustration in dimensions $(d_1,d_2)=(1,1)$ of the notations $g(k,i)$ introduced in Section \ref{s:proof-sommable}. The picture represents a detour of the path \emph{to the left} between two consecutive sites of the monotone sub-sequence $b(i_1)\prec b(i_2)$. The sequence $(g(1,1),\ldots,g(1,\ell_k))$ taking values in $\Z^{d_1}$ is equal to the \emph{horizontal} increments between two consecutive sites. Recall that we do not take into account the increment between the last blue site and $b(i_2)$.     
}\label{Fig:detour}
\begin{center}
\includegraphics[width=9cm]{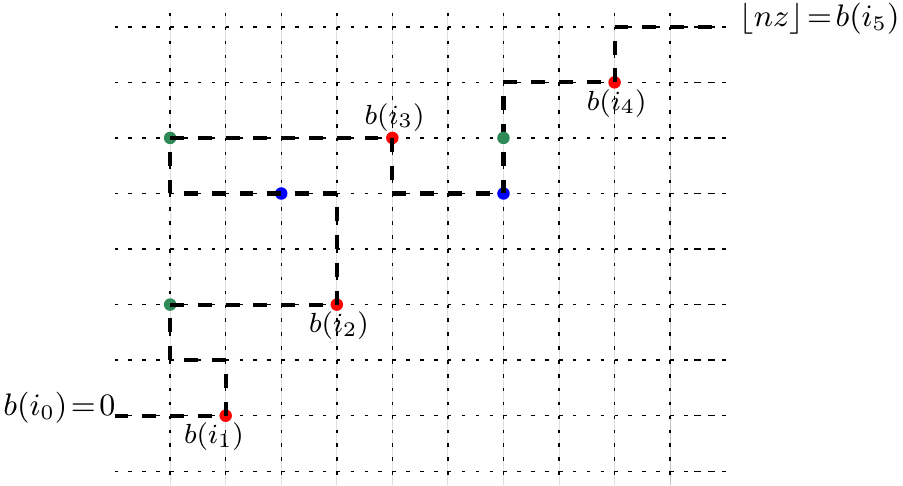}
\end{center}
\caption{Illustration of the  notation $(j_k)$ introduced in Section \ref{s:proof-sommable} in dimensions $(d_1,d_2)=(2,0)$. The blue, green and red sites represent the sequence $(b(i))_{1\le i\le  9}$  of given sites. The sites  drawn in red correspond to the monotone sub-sequence $(b(i_k))_k$. By construction, if there exists  sites between two consecutive red sites $b(i_k)$, $b(i_{k+1})$, the last one (drawn here in green) is either on the left or below (or both) of $b(i_k)$. Then,  we set $j_k=1$ if this green site is on the left of $b(i_k)$ and $j_k=2$ if not (and so, in this latter case, the green site is necessarily below $b(i_k)$). By convention, we  also set $j_k=1$ if there is no site between $b(i_k)$ and $b(i_{k+1})$, {\it i.e.}, if $\ell_k=0$.  For example, here, we have  $j_0=j_1=j_2=j_4=1$ and $j_3=2$.
 }
\label{Fig:diagonale}
\end{figure}

We just summarize here the different notations for a sequence of sites $(b(1),\ldots,b(p))$ :
\begin{itemize}
\item The sequence $(f(i))_{1 \le i \le p}$  of vectors of $\Z^{d_2}$ are the \emph{vertical} increments between two consecutive terms of the sequence.
\item The integer $q \ge 1$ is the number of terms of the monotone sub-sequence
and  for  $k \in \{0,\dots,q-1\}$,
\begin{itemize}
\item the vector $h(k)$ of  $\N^{d_1}$ is  the \emph{horizontal} increment between two consecutive terms of the monotone sub-sequence.
\item the integer $\ell_k \ge 0$ is the number of terms between two consecutive terms of the monotone sub-sequence and  the sequence $(g(k,i))_{1 \le i \le \ell_k}$  of vectors of $\Z^{d_1}$ are the \emph{horizontal} increments between these terms.
\item  the integer $j_k \in \{1,\dots,d_1\}$ is an \emph{horizontal} direction such that 
$b(i_{k+1}-1)_{j_k}\le b(i_k)_{j_k}$.
\end{itemize}
\end{itemize}

We now gather in the following lemma properties that some sequence $(b(i))$ must satisfy if the event $\cM(z,\epsilon,\eta,n)$ occurs.
\begin{lemma}\label{l:reduction-chernov}
Let $z\in(0,+\infty)^{d_1}\times\{0\}^{d_2}, \epsilon>0$ and $\eta>0$. Let $n \ge 1$ be large enough to ensure $0 \prec \lfloor nz \rfloor^-$.
On $\cM(z,\epsilon,\eta,n)$ there exist $p\ge 0$ and a sequence of open sites $b(1),\ldots,b(p)$ such that (with the notations defined above):
\begin{enumerate}
\item $0 \le n\|z\|_1 - \sum_{k=0}^{q-1}\|h(k)\|_1$.
\item $0 \le p- \sum_{i=1}^{p} \|f(i)\|_1 $.
\item $0 \le -\sum_{i=1}^{\ell_k} g(k,i)_{j_k}$ for each $k \in \{0,\dots,q-1\}$.
\item $0 \le  
\ell_k +  \sum_{1\le j \le d_1, j \neq j_k } h(k)_j-\sum_{i=1}^{\ell_k} \sum_{1\le j \le d_1, j \neq j_k }g(k,i)_j$
for each $k \in \{0,\dots,q-1\}$.
\item $0 \le 
-\sum_{k=0}^{q-1} \sum_{i=1}^{\ell_k} |g(k,i)_{j_k}| 
- \sum_{k=0}^{q-1} \sum_{1 \le j \le d_1 ,j \neq j_k} \sum_{i=1}^{\ell_k} \big(g(k,i)_j\big)_-
+ \sum_{k=0}^{q-1} \ell_k
 -\eta n \epsilon^{1/d_1}$.
 \end{enumerate}
\end{lemma}

\begin{proof}
Let $\pi$ be a nice path from $0$  to $\lfloor nz \rfloor^-$ such that $\tau_\epsilon(\pi)=T_\epsilon^{\text{nice}}(0,\lfloor nz \rfloor)$.
Such a path exists because $0\prec\lfloor nz\rfloor ^-$ (so that there exists nice paths) and $\tau_\epsilon$ takes value in $\N$
(so that the infimum is a minimum).
As $\pi$ is optimal we have (consider a path with minimal length):
\begin{equation}\label{e:geodesique1}
\tau_\epsilon(\pi) \le \|\lfloor nz \rfloor^- \|_1 - \1_{0 \text{ is open}}.
\end{equation}
We can assume that $\pi$ is self-avoiding.
Say that a vertex visited by $\pi$ -- without considering the last vertex -- is special if it is open or if it is the site $0$.
Let $(b(0), \dots,b(p))$ be the sequence -- ordered by order of visit -- of special vertices successively visited by $\pi$.
Set $b(p+1)=\lfloor nz \rfloor$.
By definition of nice paths, 
\begin{equation}
b(i) \prec \lfloor nz \rfloor \text{ for all } i \in \{0,\dots, p\}
\end{equation}
and Equation \eqref{e:becausenice} is satisfied.
Let us now prove that, on $\cM(z,\epsilon,\eta,n)$, the five items of Lemma \ref{l:reduction-chernov} are satisfied for this sequence $(b(0),\ldots,b(p+1))$.

Start with Item 1.
For each $k \in \{0,\dots,q-1\}$, as $b(i_k)\prec b(i_{k+1})$, all the coordinates of $h(k) = \projh(b(i_{k+1})-b(i_k))$ are positive.
As moreover $\sum_{k=0}^{q-1} h(k) = \projh(b(i_q)-b(i_0))=\projh(\lfloor nz\rfloor)$, we get
\[
\sum_{k=0}^{q-1}\|h(k)\|_1 = \sum_{k=0}^{q-1} \sum_{j=1}^{d_1} h(k)_j = \sum_{j=1}^{d_1} \lfloor nz\rfloor_j \le n\|z\|_1
\]
as the coordinates of $z$ are non-negative. This establishes Item 1.

Let us check Item 2.
First note the following (crude) lower bound on $\tau_\epsilon(\pi)$:
\begin{equation}\label{e:plouf}
\tau_\epsilon(\pi) \ge \|\lfloor nz \rfloor ^-\|_1 + \sum_{i=1}^{p} \|f(i)\|_1 - p - \1_{0 \text{ is open}}.
\end{equation}
Indeed the minimal length of a path from $0$ to $\lfloor nz \rfloor^- $ is $\|\lfloor nz \rfloor ^-\|_1$.
Moreover, as $0$ and $\lfloor nz \rfloor$ both belong to $\Z^{d_1}\times \{0\}^{d_2}$, each step in the last $d_2$ coordinates is a detour.
Therefore the length of $\pi$ is at least $\|\lfloor nz \rfloor ^-\|_1 + \sum_{i=1}^{p} \|f(i)\|_1$.
As $\pi$ visits $p  + \1_{0 \text{ is open}}$ open sites, \eqref{e:plouf} follows.
From \eqref{e:geodesique1} and \eqref{e:plouf} one gets Item 2.

Item 3 is just a rewriting of the definition of $j_k$ given in \eqref{e:properties_detours_1}.
Indeed recall that $g(k,i)=\projh(c(k,i)-c(k,i-1))$ so 
$$\sum_{i=1}^{\ell_k} g(k,i)_{j_k}=c(k,\ell_k)_{j_k}-c(k,0)_{j_k}=b(i_{k+1}-1)_{j_k}-b(i_k)_{j_k}\le 0.$$

Using the definition of $h(k)$ and $g(k,i)$, let us note that Item 4 is equivalent to the following inequality:
\begin{equation}\label{e:properties_detours_2}
\sum_{1\le j \le d_1, j \neq j_k } c(k,\ell_k)_j \le \ell_k + \sum_{1\le j \le d_1, j \neq j_k } c(k+1,0)_j.
\end{equation} 
If $\ell_k=0$, \eqref{e:properties_detours_2} is a consequence of \eqref{e:hi}.
Henceforth, we assume $\ell_k \ge 1$.
If $k=q-1$, \eqref{e:properties_detours_2} is straightforward by definition of nice paths.
If $k \le q-2$, \eqref{e:properties_detours_2} is due to the optimality of $\pi$.
Consider indeed the subpath $\pi_k$ of $\pi$ from $c(k,0)$ to $c(k+1,0)$.
As $\pi$ is optimal, $\pi_k$ must be optimal. 
This implies that the rewards that $\pi_k$ collects must be bigger than the additional length induced by the detour it makes to collect them. 
In other words, 
considering a path of shortest length between $c(k,0)$ and $c(k+1,0)$ (which is part of a nice path as $c(k,0) \prec c(k+1,0) \prec c(q,0)$) we get 
\begin{equation}\label{e:unsens}
\tau_\epsilon(\pi_k) \le \|c(k+1,0)-c(k,0)\|_1 - \1_{c(k,0) \text{ open}}.
\end{equation}
Using a lower bound on the length of $\pi_k$ and recalling  $\pi_k$ visits $\ell_k+\1_{c(k,0) \text{ open}}$ open sites (without counting the last one), we get
\begin{align}
\tau_\epsilon(\pi_k) 
& \ge \|c(k+1,0)-c(k,0)\|_1 + \sum_{1\le j \le d_1, j \neq j_k } \big(c(k,\ell_k)_j-c(k+1,0)_j\big)_+ - \ell_k - \1_{c(k,0) \text{ open}} \nonumber \\
& \ge \|c(k+1,0)-c(k,0)\|_1 + \sum_{1\le j \le d_1, j \neq j_k } \big(c(k,\ell_k)_j-c(k+1,0)_j\big) - \ell_k - \1_{c(k,0) \text{ open}} \label{e:lautre},
\end{align}
where $(x)_+$ denotes the positive part of $x$, \emph{i.e} $\max(x,0)$. From \eqref{e:unsens} and \eqref{e:lautre} we get \eqref{e:properties_detours_2}.

We now prove Item 5.
On $\cM(z,\epsilon,\eta,n)$, we have the following additional property:
\begin{equation}\label{e:gaindetemps1}
\tau_\epsilon(\pi) <  \|\lfloor nz \rfloor^-\|_1  - S^\cD_\epsilon(\lfloor nz \rfloor) - \eta n \epsilon^{1/d_1}.
\end{equation}
The length of $\pi$ is $\|\lfloor nz \rfloor^-\|_1$ plus the length of the detours.
Let $k \in \{0,\dots,q-1\}$ and consider the contribution to detour of the path $\pi_k$, that is the subpath of $\pi$ between $c(k,0)$ and $c(k+1,0)$.
By \eqref{e:hi}, $c(k,0) \prec \lfloor nz\rfloor$ an thus $c(k,0) \preceq \lfloor nz \rfloor^-$.
By \eqref{e:properties_detours_1} the contribution to the detour of the $j_k$-th coordinate steps of $\pi_k$ is at least
\[
\sum_{i=1}^{\ell_k} |g(k,i)_{j_k}|.
\]
By \eqref{e:hi} $c(k,0) \prec c(k+1,0)$.
Therefore the contribution to the detour of the other horizontal coordinate steps of $\pi_k$ is at least
\[
\sum_{1 \le j \le d_1 ,j \neq j_k} \sum_{i=1}^{\ell_k} \big(g(k,i)_j\big)_-
\]
where $(x)_-$ denotes the negative part of $x$, \emph{i.e.} $\max(-x,0)$.
The contribution to the detour of the vertical steps is at least
\[
\sum_{k=0}^{q-1} \|\projv(c(k+1,0)-c(k,0))\|_1.
\]
Thus the length of $\pi$ is thus at least
\begin{equation}
\label{e:ajoutM}
\|\lfloor nz \rfloor^-\|_1
+
\sum_{k=0}^{q-1} \sum_{i=1}^{\ell_k} |g(k,i)_{j_k}| + \sum_{k=0}^{q-1} \sum_{1 \le j \le d_1 ,j \neq j_k} \sum_{i=1}^{\ell_k} \big(g(k,i)_j\big)_-
+\sum_{k=0}^{q-1} \|\projv(c(k+1,0)-c(k,0))\|_1.
\end{equation}
The number of open sites visited by $\pi$ is $p+\1_{0\text{ is open}}$
(we do not consider the last site, whatever its state is).
Moreover, $p=q-1+\sum_{k=0}^{q-1} \ell_k$.
Therefore
\begin{align}
\tau_\epsilon(\pi) & \ge 
\|\lfloor nz \rfloor^-\|_1
+
\sum_{k=0}^{q-1} \sum_{i=1}^{\ell_k} |g(k,i)_{j_k}| + \sum_{k=0}^{q-1} \sum_{1 \le j \le d_1 ,j \neq j_k} \sum_{i=1}^{\ell_k} \big(g(k,i)_j\big)_-
\nonumber \\ 
& +\sum_{k=0}^{q-1} \|\projv(c(k+1,0)-c(k,0))\|_1
- q +1 -\1_{0\text{ is open}} - \sum_{k=0}^{q-1} \ell_k. \label{e:lower_bound_time}
\end{align}
Let us first consider the case where $0$ is closed.
By \eqref{e:hi}, $(c(1,0), \dots, c(q-1,0))$ belongs to $\cD_\epsilon(0,\lfloor nz \rfloor)$ and therefore
\begin{align*}
\cS^\cD_\epsilon(\lfloor nz \rfloor) 
& \ge R(c(1,0), \dots, c(q-1,0)) - V(c(1,0), \dots, c(q-1,0)) \\
& \ge q-1 - \sum_{k=0}^{q-1} \|\projv(c(k+1,0)-c(k,0))\|_1. 
\end{align*}
When $0$ is open, we use $(c(0,0), \dots, c(q-1,0)) \in \cD_\epsilon(0,\lfloor nz \rfloor)$ and  get
\[
\cS^\cD_\epsilon(\lfloor nz \rfloor) 
\ge q - \sum_{k=0}^{q-1} \|\projv(c(k+1,0)-c(k,0))\|_1. 
\]
In all cases,
\begin{equation}\label{e:upper_time_score}
\cS^\cD_\epsilon(\lfloor nz \rfloor) 
 \ge q+\1_{0\text{ is open}}-1 - \sum_{k=0}^{q-1} \|\projv(c(k+1,0)-c(k,0))\|_1. 
\end{equation}
Item 5 follows from \eqref{e:gaindetemps1}, \eqref{e:lower_bound_time} and \eqref{e:upper_time_score}.
\end{proof}

Lemma \ref{l:reduction-chernov} prepares the application of union bounds and 
deterministic Chernov inequalities in the proof of Lemma \ref{l:sommable}.
As usual, the strategy is to keep enough information to avoid combinatorial explosion while throwing enough information to get manageable expressions.
Note for example that Item 2 already provides a good control on the length of the vertical increments. For that reason, we do not have to take all their contributions into account in Items 4 and 5, especially when giving a lower bound on the length of $\pi$ in \eqref{e:lautre} and \eqref{e:ajoutM} - which is good news since it would have add some pretty messy terms.
As another example, note that Items 3-5 give a good control on one coordinate of the $g(k,i)$ but a cruder one on the other coordinates.
This gives some relative freedom on $d_1-1$ coordinates of the $g(k,i)$.
Because of this, for example, there is an exploding factor of order $\epsilon^{-(d_1-1)/d_1}$ in \eqref{e:clete}.
This is however harmless because it is multiplied by $\epsilon$.

\begin{proof}[Proof of Lemma \ref{l:sommable}]
For all $s>0$ we introduce
\[
K(s) = \sum_{a \in \Z} \exp(- s|a|)=\frac{1+e^{-s}}{1-e^{-s}} \in [1,\infty).
\]
We will use in particular the following properties of $K$:
\begin{equation}\label{e:k:properties}
s \mapsto K(s) \text{ is decreasing}, K(s) \to 1 \text{ as } s \to \infty, K(s) \sim \frac 2 s \text{ as } s \to 0.
\end{equation}

Let $z,\epsilon,\eta$ and $n$ be as in Lemma \ref{l:reduction-chernov}.
Write $\cM(n)=\cM(z,\epsilon,\eta,n)$ for short.
We give an upper bound on $\P[\cM(n)]$ using Lemma \ref{l:reduction-chernov}, union bound and deterministic Chernov inequality 
(that is bound of an indicator by an exponential). On the event $\cM(n)$, Lemma \ref{l:reduction-chernov} gives us the existence of a $p\geq 0$ and of a sequence $(b(1), \dots , b(p))$, thus of the corresponding $q, (\ell_k)_k, (j_k)_k, (f(i))_i, (g(k,i))_{k,i}$ and $(h(k))_{k}$ satisfying Items 1 to 5.
By a union bound, we thus get a sum over $q, (\ell_k)_k, (j_k)_k, (f(i))_i, (g(k,i))_{k,i}$ and $(h(k))_{k}$ (as in Lemma \ref{l:reduction-chernov}) of
\begin{equation}\label{e:un_terme}
\epsilon^{q-1 + \sum \ell_k} \1_{\text{Item }1}\1_{\text{Item }2}\1_{\text{Item }4}\1_{\text{Item }5}
\end{equation}
where $\1_{\text{Item } 1}=1$ if and only if the condition described in Item 1 holds and so on. The term $\epsilon^{q-1 + \sum \ell_j}=\epsilon^p$ is simply the probability that a fixed sequence of sites $(b(1),\ldots,b(p))$ are open. Note that we bound $\1_{\text{Item } 3}$ by $1$ (we kept Item 3 for clarity in the statement of Lemma \ref{l:reduction-chernov} but it is useless here). 
We now fix $\alpha, \beta, \gamma >0$ such that
\begin{equation} \label{e:definition_constantes}
\beta=2\gamma, \quad eK(1)^{d_2} 2^{d_1}\beta^{-1}\gamma^{-(d_1-1)}= \frac 1 {8d_1} \text{ and } \alpha\eta-\beta \|z\|_1=1.
\end{equation}
The constants $\alpha,\beta,\gamma$ depends only on $z$ and $\eta$ (recall $d_1=d_1(z)$ and $d_2=d_2(z)$).
There exists $\epsilon_0=\epsilon_0(z,\eta)>0$ such that, for all $\epsilon<\epsilon_0$,
\begin{equation} \label{e:epsilon_condition}
 \gamma \epsilon^{1/d_1}<\alpha
\end{equation}
holds. We make this assumption in the remaining of the proof.
We use
\begin{align*}
\1_{\text{Item } 1} & \le M_1 := \exp[\beta \epsilon^{1/d_1} n\|z\|_1]\prod_{k=0}^{q-1}\exp[ -\beta  \epsilon^{1/d_1}\|h(k)\|_1], \\
\1_{\text{Item } 2} & \le M_2 := \prod_{i=1}^{q-1+\sum \ell_k} \exp[1-\|f(i)\|_1], \\
\1_{\text{Item } 4} & \le M_4 := \prod_{k=0}^{q-1} 
\exp\Big[
\gamma\epsilon^{1/d_1}\Big(\ell_k +  \sum_{1\le j \le d_1, j \neq j_k } h(k)_j-\sum_{i=1}^{\ell_k} \sum_{1\le j \le d_1, j \neq j_k }g(k,i)_j\Big)
\Big], \\
\1_{\text{Item } 5} &  \le M_5 :=
\exp[-\alpha \eta n \epsilon^{1/d_1}]
\prod_{k=0}^{q-1} 
\exp\Big[
- \alpha \sum_{i=1}^{\ell_k} |g(k,i)_{j_k}| 
-  \alpha \sum_{1 \le j \le d_1 ,j \neq j_k} \sum_{i=1}^{\ell_k} \big(g(k,i)_j\big)_-
+  \alpha \ell_k
\Big].
\end{align*}

Fix $q \ge 1, \ell_0, \dots, \ell_{q-1} \ge 0$ and $j_0,\dots,j_{q-1} \in \{1,\dots,d_1\}$.
We first focus on the sum $W(q,\ell_\cdot,j_\cdot)$ of \eqref{e:un_terme} over all the other variables:
\begin{equation}\label{e:W_somme_partielle}
W(q,\ell_\cdot,j_\cdot) = \sum_{\text{the } f(i), g(k,i) \text{ and }h(k)} \eqref{e:un_terme}.
\end{equation}
For all $k \in \{0, \dots, q-1\}$,  denote by $h'(k) \in \N$ the $j_k$-th coordinate of $h(k)$ and by $h''(k) \in \N^{d_1-1}$ the other coordinates.
In other words,
\[
h'(k)=h(k)_{j_k} \in \N \text{ and } h''(k) = (h(k)_j)_{1 \le j \le d_1, j \neq j_1} \in \N^{d_1-1}.
\]
Define similarly the $g'(k,i) \in \Z$ and the $g''(k,i) \in \Z^{d_1-1}$.
We can see the sum $W(q,\ell_\cdot,j_\cdot)$ as a sum over the new variables:
\begin{equation}\label{e:W_somme_partielle_prime}
W(q,\ell_\cdot,j_\cdot) = \sum_{\text{the } f(i), g'(k,i), g''(k,i), h'(k) \text{ and }h''(k)} \eqref{e:un_terme}.
\end{equation}
We can also see \eqref{e:un_terme} as a product of a constant, 
a factor depending on the $f(i)$, a factor depending on the $g'(k,i)$ and so on.
We will consider separately the sum of each kind of factor over the variables on which it depends.

Let start with the factor depending on the $f(i)$ (for simplicity we consider the whole $M_2$).
For short we write $p=q-1+\sum_k \ell_k$.
The sum is over $f(1), \dots, f(p) \in \Z^{d_2}$.
We get
\begin{align*}
\sum_{\text{the }f(i) \in \Z^{d_2}} \prod_{i=1}^p \exp[1-\|f(i)\|_1] 
& = \Big(\sum_{u \in \Z^{d_2}}  \exp[1-\|u\|_1]\Big)^p \\
& = \left(eK(1)^{d_2}\right)^p\\
& \le  \left(eK(1)^{d_2}\right)^{q+\sum_k \ell_k}.
\end{align*}
We now deal with the factor depending on the $g'(k,i)$:
\begin{align*}
\sum_{\text{the }g'(k,i) \in \Z} \prod_{k=0}^{q-1} \exp\Big[- \alpha \sum_{i=1}^{\ell_k} |g'(k,i)| \Big] = K(\alpha)^{\sum_k \ell_k}.
\end{align*}
We go on with the next factors. We get
\begin{align*}
& \sum_{\text{the }g''(k,i) \in \Z^{d_1-1}} 
\prod_{k=0}^{q-1} 
\exp\Big[ - \gamma\epsilon^{1/d_1}\sum_{i,j} g''(k,i)_j - \alpha \sum_{i,j}  \big(g''(k,i)_j\big)_-\Big] \\
 & \quad =  \Big( \sum_{u \in \Z}  \exp[ - \gamma\epsilon^{1/d_1} u - \alpha u_-]\Big)^{(d_1-1)\sum_k \ell_k} \\
 & \quad\le  \big(K(\gamma\epsilon^{1/d_1})+K(\alpha-\gamma\epsilon^{1/d_1})\big)^{(d_1-1)\sum_k \ell_k}
\end{align*}
where we used \eqref{e:epsilon_condition} in the last step.
Concerning the $h'(k)$ we have
\begin{align*}
\sum_{\text{the }h'(k) \in \Z} 
\prod_{k=0}^{q-1}\exp[ -\beta  \epsilon^{1/d_1}|h'(k)|]=K(\beta \epsilon^{1/d_1})^q.
\end{align*}
Using $\beta=2\gamma$ (see \eqref{e:definition_constantes}) we obtain
\begin{align*}
& \sum_{\text{the }h''(k) \in \Z^{d_1-1}} 
\prod_{k=0}^{q-1}\exp\Big[ -\beta  \epsilon^{1/d_1}\|h''(k)\|_1 + \gamma \epsilon^{1/d_1} \sum_j h''(k)_j\Big]  \\
& \quad = 
\sum_{\text{the }h''(k) \in \Z^{d_1-1}} 
\prod_{k=0}^{q-1}\exp\Big[ -2\gamma  \epsilon^{1/d_1}\|h''(k)\|_1 + \gamma \epsilon^{1/d_1} \sum_j h''(k)_j\Big]\\
& \quad =  
\sum_{\text{the }h''(k) \in \Z^{d_1-1}} 
\prod_{k=0}^{q-1}\exp\Big[ -\gamma  \epsilon^{1/d_1}\|h''(k)\|_1 \Big] \\
& \quad = K(\gamma  \epsilon^{1/d_1})^{q(d_1-1)}.
\end{align*}
Coming back to \eqref{e:W_somme_partielle_prime}, 
using the above upper bounds 
and rearranging the factors (in particular we distribute the power of $\epsilon$ in the different factors),
 we get:
\begin{align}\label{e:borneW}
W(q,\ell_\cdot,j_\cdot) =
& \frac 1 \epsilon \exp\Big[\beta \epsilon^{1/d_1} n\|z\|_1  -\alpha \eta n \epsilon^{1/d_1}\Big] \\ \nonumber
& \left(\epsilon eK(1)^{d_2}K(\beta \epsilon^{1/d_1})K(\gamma  \epsilon^{1/d_1})^{(d_1-1)}\right)^{q} \\ \nonumber
& \left(\epsilon \exp[\gamma \epsilon^{1/d_1} + \alpha ]
eK(1)^{d_2}K(\alpha)\big(K(\gamma\epsilon^{1/d_1})+K(\alpha-\gamma\epsilon^{1/d_1})\big)^{d_1-1}\right)^{\sum_k \ell_k}. 
\end{align}
Using \eqref{e:k:properties} we get, as $\epsilon \to 0 $,
\[
K(\gamma\epsilon^{1/d_1})+K(\alpha-\gamma\epsilon^{1/d_1}) \sim 2\gamma^{-1}\epsilon^{-1/d_1}
\]
and thus, for some constant $C=C(\alpha,\beta,\gamma,d_1,d_2)$,
\begin{equation}\label{e:clete}
\epsilon \exp[\gamma \epsilon^{1/d_1} + \alpha ]
eK(1)^{d_2}K(\alpha)\big(K(\gamma\epsilon^{1/d_1})+K(\alpha-\gamma\epsilon^{1/d_1})\big)^{d_1-1}
\sim C \epsilon^{1/d_1}
\end{equation}
as $\epsilon$ tends to 0.
In particular there exists $\epsilon_1=\epsilon_1(z,\eta)$ (recall that $\alpha, \beta, \gamma$ only depends on $z$ and $\eta$
and that $d_1=d_1(z)$, $d_2=d_2(z)$)
such that for all $\epsilon<\epsilon_1$, the above quantity is at most $1/2$.
Using \eqref{e:k:properties} we also get, as $\epsilon \to 0$,
\[
\epsilon eK(1)^{d_2}K(\beta \epsilon^{1/d_1})K(\gamma  \epsilon^{1/d_1})^{(d_1-1)} \to eK(1)^{d_2} 2^{d_1}\beta^{-1}\gamma^{-(d_1-1)}.
\]
By \eqref{e:definition_constantes}, this limit is equal to $1/(8d_1)$. 
Therefore there exists $\epsilon_2=\epsilon_2(z,\eta)$ 
such that for all $\epsilon<\epsilon_2$, the above quantity is at most $1/(4d_1)$.
Set $\epsilon_3=\min(\epsilon_0,\epsilon_1,\epsilon_2)$, which depends only on $z$ and $\eta$.
For all $\epsilon<\epsilon_3$ we thus have
\[
W(q,\ell_\cdot,j_\cdot) \le 
 \frac 1 \epsilon \exp\Big[\beta \epsilon^{1/d_1} n\|z\|_1  -\alpha \eta n \epsilon^{1/d_1}\Big] \left( \frac 1 {4d_1}  \right)^q
 \left(\frac 1 2 \right)^{\sum_k \ell_k}.
\]
Therefore (note that the above expression does not depends on the $j_k$)
\[
\sum_{\text{the }j_k} W(q,\ell_\cdot,j_\cdot) \le 
 \frac 1 \epsilon \exp\Big[\beta \epsilon^{1/d_1} n\|z\|_1  -\alpha \eta n \epsilon^{1/d_1}\Big] \left( \frac 1 4  \right)^q
 \left(\frac 1 2 \right)^{\sum_k \ell_k}
\]
and then 
\[
\sum_{\text{the }\ell_k} \sum_{\text{the }j_k} W(q,\ell_\cdot,j_\cdot) \le 
 \frac 1 \epsilon \exp\Big[\beta \epsilon^{1/d_1} n\|z\|_1  -\alpha \eta n \epsilon^{1/d_1}\Big] \left( \frac 1 2  \right)^q
\]
and therefore
\[
\P[\cM(n)] \le \sum_{q \ge 1} \sum_{\text{the }\ell_k} \sum_{\text{the }j_k }W(q,\ell_\cdot,j_\cdot)
 \le \frac 1 \epsilon \exp\Big[\beta \epsilon^{1/d_1} n\|z\|_1  -\alpha \eta n \epsilon^{1/d_1}\Big].
\]
By \eqref{e:definition_constantes} we thus have, for all $\epsilon<\epsilon_3$,
\[
\P[\cM(n)]\le \frac 1 \epsilon  \exp\Big[- \epsilon^{1/d_1} n\Big].
\]
This proves the lemma. 
\end{proof}

\section{Proof of Theorem \ref{t2}: the bond case}
\label{s:bond}

The proof in the case of bond percolation uses the same strategy as before. First, we study an oriented model which is equivalent in the limit to a semi-continuous model. Then, we prove that, at first order, the oriented and non-oriented model are equal. 
In this section, we just highlight the differences compared to the case of the site percolation.
As in the introduction, we use bold letters for objects in the framework of Bernoulli bond first-passage percolation.

\subsection{A related discrete oriented model}\label{a:discret_model}

Recall the relation introduced in Section \ref{s:discrete_model} between the sites of $\Z^d$ :
for $x, y \in \Z^d$, we write  
\[
x \prec y \text{ if for all }  i \in \{1,\dots,d_1\}, \; x_i < y_i.
\]
We extend this relation for edges of $\Z^d$ in the following way. Denote by $(e_1,\ldots,e_d)$ the canonical basis of $\Z^d$. Then  for $u:=(x,x+e_i)$ and $v:=(y,y+e_j)$ two edges of $\Z^d$, we write
$$u\prec v \text{ if  }   \; x \prec y.$$
We will also use this relation to compare a site $x$ to an edge $v:=(y,y+e_j)$ with the convention $x\prec v$ (resp.~ $x \preceq v$, $v\prec x$) if $x\prec y$ (resp.~$x \preceq y$, $y\prec x$).
We now define the set of monotone sequences of open edges between sites $x\preceq y$ in $\Z^{d_1}\times \{0\}^{d_2}$ by 
\begin{align*}
& \bcD_\epsilon(x,y) = \\
& \quad \{ (\bw(1),\dots,\bw(k)) : k \ge 0, \bw(1),\dots, \bw(k)  \text{ are open edges of } \Z^d 
\text{ such that } x \preceq \bw(1) \prec \cdots \prec \bw(k) \prec y\}.
\end{align*}
Notice that the notation $\bw(i)$ refers to an edge, whereas the notation $w(i)$ we used previously referred to a point in $\mathbb{Z}^d$. 
Let $\alpha(i),\beta(i)$ be the two extremities of the edge $\bw(i)$ and define $A(i)$ as
\begin{itemize}
\item The edge $\{\projv(\alpha(i)), \projv(\beta(i)) \}$ if  $\projv(\alpha(i))\neq\projv(\beta(i))$, \emph{i.e.}, if $\bw(i)$ is a \emph{vertical} edge.
\item The point $\projv(\alpha(i) ) $ otherwise, \emph{i.e.}, if $\bw(i)$ is an \emph{horizontal} edge.
\end{itemize}
We denote $\bV(\bw(1),\dots,\bw(k))$ the $\ell_1$-length of the shortest path in $\Z^{d_2}$ from  $0$ to $0$ that goes  successively through  the (unoriented) edges and sites $A(1), . . . , A(k)$ in this order. Note that, since $(\bw(1),\dots,\bw(k))\in  \bcD_\epsilon(x,y)$,   $\|y-x\|_1+ \bV(\bw(1),\dots,\bw(k))$ is  in fact equal to the $\ell_1$-length of the   shortest path from $x$ to $y$  going through the edges $\bw(1),\dots,\bw(k)$ and $\bV(\bw(1),\dots,\bw(k))$ is the total vertical displacement of this path.

Let $\bR(\bw(1),\dots,\bw(k)):=k$ be the number of rewards collected along this path.
As before, we define the score $\bS^\cD_\epsilon(x,y)$  between $x \preceq y$ in $\Z^{d_1} \times \{0\}^{d_2}$ by 
\begin{equation*}
\bS^\cD_\epsilon(x,y) =  \sup_{s \in \bcD_\epsilon(x,y)} \big(\bR(s) - \bV(s)\big).
\end{equation*}
An example is given in Figure \ref{Fig:discret_arete}.
Note that the analogs of Lemma \ref{l:elementary} and Lemma \ref{l:sigma_discrete} clearly also hold in this setting
and we can define a mean directional score $\bsigma^\cD_\epsilon(\cdot)$ for all $z \in (0,+\infty)^{d_1} \times \{0\}^{d_2}$ by
\begin{equation*}
\bsigma^\cD_\epsilon(z) : = \lim_{n \to \infty} \frac 1 n \E[\bS^\cD_\epsilon(0,\lfloor nz \rfloor )].
\end{equation*}

\begin{figure}
\begin{center}
\includegraphics[width=10cm]{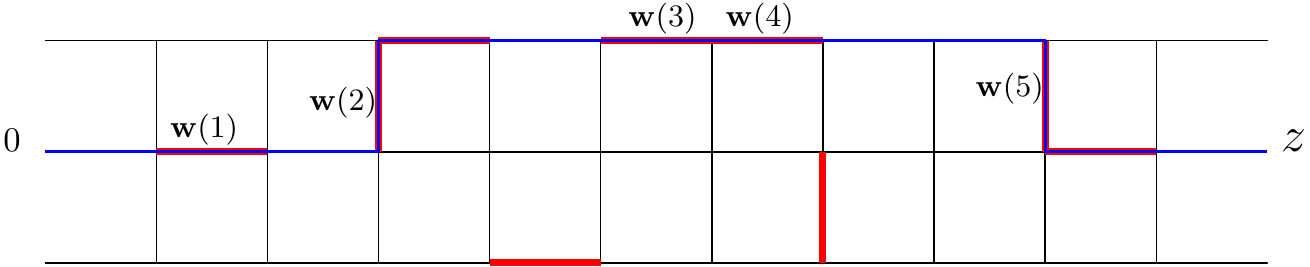}
\end{center}
\caption{Example in the discrete oriented model of bond percolation in dimensions $(d_1,d_2) = (1,1)$. Open edges are drawn in red. Here, the score $\bS^\cD_\epsilon(0,z)$ is achieved using the path drawn in blue. Note that two open edges $\bw(i),\bw(i+1)$ taken by the blue path are counted as two rewards for this path  only if $\bw(i)\prec \bw(i+1)$. So here, we have $\bS^\cD_\epsilon(0,z)=5-2=3$. } 
\label{Fig:discret_arete}
\end{figure}

\subsection{A related semi-continuous oriented model}
\label{a:continuous_model}

Let $(\xi_1,\ldots,\xi_d)$ be $d$ independent Poisson point processes on $\R^{d_1} \times \Z^{d_2}$ with intensity $\nu=\d x \otimes \d v$, where $\d x$ denotes the Lebesgue measure on $\R^{d_1}$ and $\d v $ the counting measure on $ \Z^{d_2}$.
We call {\em particles of type $j$}   the points  of $\xi_j$.
Let $\epsilon>0$.
For $x=(a,b) \in \Z^{d_1+d_2}$, we define the $\epsilon-$cube $C_\epsilon(x)= C_\epsilon (a,b)$ as in \eqref{eq:def_Cube}. Let $\tilde\epsilon:=1-\exp(-\varepsilon)$ and for an edge $u:=(x,x+e_j)\in \mathbb{E}^d$,  set
 \[
\btau_{\tilde\epsilon}(u) = \1_{\{C_\epsilon(x)\cap \xi_j=\emptyset\}}.
\]
Hence, the edge $(x,x+e_j)$ is open if and only if there is a particle of type $j$ in the $\varepsilon$-cube $C_\varepsilon(x)$.
The family $(\btau_{\tilde\epsilon}(u))_{u \in \mathbb{E}^d }$ is again a family of independent Bernoulli random variables with parameter $1-\tilde\epsilon$.
We use this family to define a bond percolation on $\Z^d$ with  parameter $\tilde\epsilon$.

Let us now explain how we construct the semi-continuous oriented model on $\R^{d_1} \times \Z^{d_2}$ and define the score between two sites. Let $\Xi:=\cup_{i=1}^d \xi_i$ be the union of the $d$ Poisson point processes. Recall the definition \eqref{e:ajoutM3} of the relation $\prec_{\alpha}$ introduced in Section \ref{s:continuous_model} and consider as before, for any $z \in (0,+\infty)^{d_1}\times \{0\}^{d_2}$, the set
$\bcC_\alpha(z)$  of all monotone sequences of particles of $\Xi$ between $0$ and $z$ which are $\alpha$-separated, which we define as the set
\[
\bcC_\alpha(z)=
\{ (w(1),\dots,w(k)) : k \ge 0, w(1),\dots, w(k)  \in \Xi
\text{ such that } 0 \preceq w(1) \prec_\alpha \cdots \prec_\alpha w(k) \prec_\alpha z\}.
\]
 
Let us now emphasize here a difference compared to the case of site percolation for the definition of the score. Particles of $\xi_j$ will have a distinct role  whether $j\le d_1$ or $j>d_1$. Indeed
\begin{itemize}
\item For $j\le d_1$, a particle $w=(x,\nu)\in \R^{d_1} \times \Z^{d_2}$  of type $j$  creates a reward at position $(x,\nu)$.
\item For $j>d_1$, a particle $w=(x,\nu)\in \R^{d_1} \times \Z^{d_2}$ of type $j$ creates an open \emph{vertical} edge from $(x,\nu)$ to $(x,\nu)+e_j$.
\end{itemize}
For a sequence $(w(1),\dots,w(k))\in \bcC_\alpha(z)$ of points of $\Xi$, let define $A(i)$  for $1\le i \le k$ as
\begin{itemize}
\item The edge $\{\projv(w(i)), \projv(w(i)+e_j) \}$ if  $w(i)$ is a particle of type $j>d_1$, \emph{i.e.}, if $w(i)$ is associated to a \emph{vertical} edge.
\item The point $\projv(w(i) ) $ otherwise, \emph{i.e.}, if $w(i)$ is associated to a reward.
\end{itemize}
As before, we denote $\bV(w(1),\dots,w(k))$ the $\ell_1$-length of the shortest path in $\Z^{d_2}$ from $0$ to $0$ that goes  successively through  the (unoriented) edges and sites $A(1), . . . , A(k)$ in this order. Note that we commit here a slight abuse of notation since $\bV(w(1),\dots,w(k))$ does not depend only on $(w(1),\dots,w(k))$ but also on the types of particles associated to each term of the sequence $(w(1),\dots,w(k))$. 
%
%
Let us remark that going through an edge is more constraining than going through a site, and so the path in $\Z^{d_2}$ associated here to a sequence  $(w(1),\ldots,w(k))$ of points of $\Xi$ has an $\ell_1$-length greater than or equal to the one defined for the same sequence of points in Section \ref{s:continuous_model}, \emph{i.e.},
 \begin{equation}\label{e:inegaliteV}
 \bV(w(1),\ldots,w(k))\ge V(w(1),\ldots,w(k)).
 \end{equation}
 Finally, let $\bR(w(1),\ldots,w(k)):=k$ as before.

Then,  for all $z \in (0,+\infty)^{d_1} \times \{0\}^{d_2}$ and any $\alpha \ge 0$,  we  define the score $\bS^{\cC}_\alpha(z)$  by
\[
\bS^{\cC}_\alpha(z)= \sup_{s \in \bcC_\alpha(z)}\big(\bR(s)-\bV(s))
\]
(see an example in Figure \ref{Fig:ppp_arete}) and the  mean directional score $\bsigma^\cC_\alpha$ as in Lemma \ref{l:sigmac} by
\[
\bsigma^\cC_\alpha(z) 
:= \lim_{\lambda \to \infty} \frac1 \lambda \E\big[\bS^{\cC}_\alpha(\lambda z)\big] 
= \sup_{\lambda >0} \frac1 \lambda \E\big[\bS^{\cC}_\alpha(\lambda z)\big].
\]

\begin{figure}
\begin{center}
\includegraphics[width=10cm]{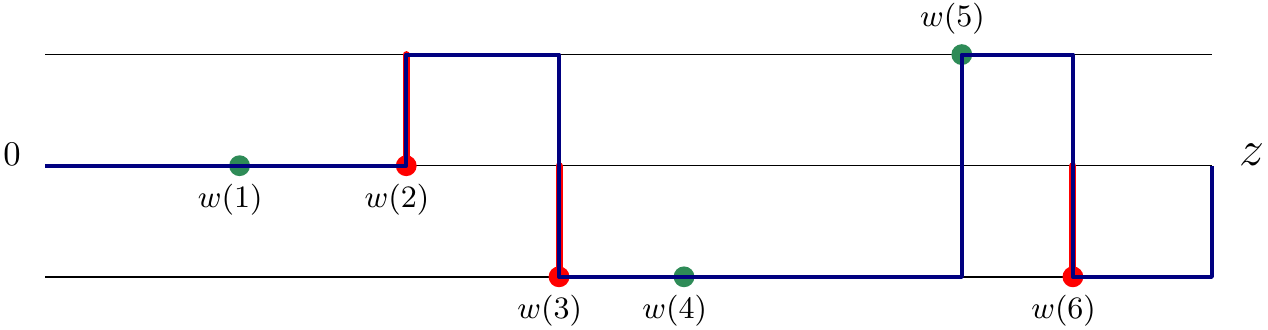}
\end{center}
\caption{Example of a path in the semi-continuous model in dimensions $(d_1,d_2) = (1,1)$. Particles of type $1$ are drawn in green whereas particles of type $2$ are drawn in red with their associated edge. In blue, a path with minimal length going through $s=(w(1),\ldots,w(6))$. Here $\bR(s)-\bV(s)=6-8=-2.$} 
\label{Fig:ppp_arete}
\end{figure}

All the arguments explained in Section \ref{s:link_discret_continu} still hold in this setting and so we also get 
\begin{lemma}\label{l:link_discrete_continous_arete} For all $z \in (0,+\infty)^{d_1} \times \{0\}^{d_2}$,
\[
\lim_{\tilde\epsilon \to 0} \frac{\bsigma^\cD_{\tilde\epsilon} (z)}{\tilde\epsilon^{1/{d_1}}} = \bsigma_0^{\cC}(z).
\]
\end{lemma}
Moreover, using scaling argument, we also get that for all $z \in (0,+\infty)^{d_1} \times \{0\}^{d_2}$, 
\begin{equation}
\label{e:ajoutM2}
\bsigma^\cC_0(z) = \gamma(z) \bsigma^\cC_0(\1_{d_1}).
\end{equation}
Besides, by a coupling argument, we get the following bound on $\bsigma^\cC_0(\1_{d_1})$:
\begin{lemma}\label{l:encadrement_arete}For all $d_1\ge 1$ and $d_2\ge 0$,
$$
d_1^{1/d_1}\sigma^\cC_0(\1_{d_1})\le \bsigma^\cC_0(\1_{d_1})\le  (d_1+d_2)^{1/d_1}\sigma^\cC_0(\1_{d_1}).$$
\end{lemma}

\begin{rem}
Recall that $\sigma^\cC_0(\1_{d_1})$ (resp. $\bsigma^\cC_0(\1_{d_1})$) is the directional score in direction $\1_{d_1}$ in the semi-continuous model of dimension $d_1+d_2$ associated with site (resp. bond) first-passage percolation. So, $\sigma^\cC_0(\1_{d_1})$ and $\bsigma^\cC_0(\1_{d_1})$ also depend on $d_2$ although it is not recalled in our notation.
\end{rem}

\begin{proof} Recall that, in the case of the site percolation, the semi-continuous model (and so the directional score   $\sigma^\cC_0(\1_{d_1})$) is constructed from a Poisson point process  on $\R^{d_1} \times \Z^{d_2}$ with intensity $\nu=\d x \otimes \d v$. By a scaling argument, if we instead consider a Poisson point process on $\R^{d_1} \times \Z^{d_2}$ with intensity $\lambda \nu$ to construct the semi-continuous model, we easily see that the directional score in direction $\1_{d_1}$ will be equal then to $\lambda^{1/d_1}\sigma^\cC_0(\1_{d_1})$.
Besides, recall that the semi-continuous model for bond percolation is constructed from a collection of $d=d_1+d_2$ Poisson point processes $(\xi_1,\ldots,\xi_d)$. 

Then, on one hand, we use that, in the definition of the score of bond percolation, for $j\le d_1$, the points of $\xi_j$ corresponds to an open site   on $\R^{d_1}\times \Z^{d_2}$. Hence, noticing that $\Xi_{d_1}:=\cup_{j\le d_1} \xi_j$ is just a Poisson point process   on $\R^{d_1} \times \Z^{d_2}$ with intensity $d_1\nu$ and that adding vertical open edges can only increase the score, we deduce that 
$$\bsigma^\cC_0(\1_{d_1})\ge d_1^{1/d_1}\sigma^\cC_0(\1_{d_1}).$$

And on the other hand, a path going through an edge $(x,x+e_j)$ necessarily goes through the site $x$ 
and so the scores in the semi-continuous model of  bond percolation are necessarily smaller than or equal to the scores in the semi-continuous model of site percolation constructed from the Poisson point process $\Xi:=\cup_{j\le d} \xi_j$ which have intensity  $(d_1+d_2)\nu$ (\emph{c.f.} Eq. \eqref{e:inegaliteV}).
Thus, we get 
$$\bsigma^\cC_0(\1_{d_1})\le (d_1+d_2)^{1/d_1}\sigma^\cC_0(\1_{d_1}).$$
\end{proof}

\subsection{Link between the oriented model and the original model}

\label{a:link_oriented_original}

As in the case of site percolation, it remains to prove that, for $\varepsilon$ small enough, the discrete oriented model  and the non-oriented model have the same behavior, more precisely that we have the analogue of Proposition \ref{l:reduction_oriented}:

\begin{lemma} \label{a:reduction_oriented} 
For all $z\in(0,+\infty)^{d_1}\times\{0\}^{d_2}$,
\[
\lim_{\epsilon \to 0} \frac{\bmu_\epsilon(z)-(\|z\|_1-\bsigma_\epsilon^\cD(z))}{\epsilon^{1/d_1}} = 0.
\]
\end{lemma}

The proof of Lemma \ref{a:reduction_oriented} is quite technical since we will again define 
for all $z \in (0,+\infty)^{d_1} \times \{0\}^{d_2}$, $\epsilon>0$, $\eta>0$ and $n$  the event
\[
\bcM(z,\epsilon,\eta,n)  = \{\bT^{\text{nice}}_\epsilon(0, \lfloor nz \rfloor) < \|\lfloor nz \rfloor^-\|_1 
- \bS^\cD_\epsilon(\lfloor nz \rfloor) - \eta n \epsilon^{1/d_1}\}
\]
and prove that 
 for all $\epsilon>0$ small enough (depending on $z$ and $\eta$),
\begin{equation}\label{e:M_araete}
\lim_{n\to \infty} \P[\bcM(z,\epsilon,\eta,n)] = 0.
\end{equation}

To prove this, we adapt the proof of Lemma \ref{l:sommable} to our new setting. We first say that a site $x\in \Z^d$ is \emph{open} if one of the $d$ edges $(x,x+e_j)$ is open for $(e_1,\ldots,e_d)$ the canonical basis of $\Z^d$. Note that, with this convention, each site is open independently of the others and with probability
$$\bar{\varepsilon}:=1-(1-\varepsilon)^d\sim d\varepsilon.$$
Moreover, we say that a site $x\in \Z^d$ is \emph{doubly open} if at least two edges $(x,x+e_j)$ with $j\in\{1,\ldots,d\}$ are open.
Hence, a site is doubly open with probability
$$1-(1-\varepsilon)^d-d(1-\varepsilon)^{d-1}\varepsilon\sim \frac{d(d-1)}{2}\varepsilon^2\le \bar{\varepsilon}^2 ,$$
for $\epsilon$ small enough.
Consider an optimal path $\pi$ between $0$ and $\lfloor nz \rfloor$ which is self avoiding.
Let $b(1),\ldots, b(p)$ be the open sites it goes through and, as in Section \ref{s:proof-sommable}, denote by $(b(i_k))_{0\le k\le q}$ the associated monotone sub-sequence and $\ell_k=i_{k+1}-i_k-1$ the number of open sites between two consecutive terms of this monotone sub-sequence. Denote also by
 $m$ the number of 
\emph{doubly open} sites in the sub-sequence $(b(i_k))_{1\le k\le q-1}$. Let $\pi_k$ be the restriction of $\pi$ from $b(i_k)$ to $b(i_{k+1})$. Let us remark that if an edge is open, by definition, one of its extremity is open. Moreover, any given site is an extremity of at most two edges of a given self-avoiding path. Thus,
without counting the first and last step, the number of open edges $\pi_k$ goes through is at most $2\ell_k$ (and $2(\ell_k+1)$ if we also count the first and last step). Besides,  $q+1+m$ is an upper bound on the number of open edges of the path touching one of the sites $(b(i_k))_{1\le k\le q-1}$. 
Hence, the total number of open edges $\pi$ goes trough is at most 
$$q+m+1+2\sum_{k=0}^{q-1} \ell_k\le 2(p+1).$$

From these remarks and using the same arguments as in the proof of Lemma \ref{l:reduction-chernov}, we can then easily establish the following result.

\begin{lemma}\label{l:reduction-chernov2}
Let $z\in(0,+\infty)^{d_1}\times\{0\}^{d_2}, \epsilon>0$ and $\eta>0$. Let $n \ge 1$ be large enough to ensure $0 \prec \lfloor nz \rfloor^-$.
On $\bcM(z,\epsilon,\eta,n)$ there exists $p,m\ge 0$ and  a sequence of open sites $b(1),\ldots,b(p)$ such that $m$ sites of the monotone sub-sequence $(b(i_k))_{1\le k\le q-1}$ are \emph{doubly open} and :
\begin{enumerate} 
\item $0 \le n\|z\|_1 - \sum_{k=0}^{q-1}\|h(k)\|_1$.
\item $0 \le 2(p+1)- \sum_{i=1}^{p} \|f(i)\|_1 $.
\item $0 \le -\sum_{i=1}^{\ell_k} g(k,i)_{j_k}$ for each $k \in \{0,\dots,q-1\}$.
\item $0 \le  
2(\ell_k+1) +  \sum_{1\le j \le d_1, j \neq j_k } h(k)_j-\sum_{i=1}^{\ell_k} \sum_{1\le j \le d_1, j \neq j_k }g(k,i)_j$
for each $k \in \{0,\dots,q-1\}$.
\item $0 \le 
-\sum_{k=0}^{q-1} \sum_{i=1}^{\ell_k} |g(k,i)_{j_k}| 
- \sum_{k=0}^{q-1} \sum_{1 \le j \le d_1 ,j \neq j_k} \sum_{i=1}^{\ell_k} \big(g(k,i)_j\big)_-
+m+1+ 2\sum_{k=0}^{q-1} \ell_k
 -\eta n \epsilon^{1/d_1}$.
 \end{enumerate}
\end{lemma}
The differences with Lemma \ref{l:reduction-chernov} are that, in Item 2, a $p$ becomes a $2(p+1)$, in Item 4,  $\ell_k$ becomes $2(\ell_k+1)$ and in Item 5, $\sum_{k=0}^{q-1} \ell_k$ becomes $m+1+ 2\sum_{k=0}^{q-1} \ell_k$.

Now, the probability that a fixed sequence of sites  $b(1),\ldots,b(p)$ are open and   $m$ sites of the monotone sub-sequence $(b(i_k))_{1\le k\le q-1}$ are \emph{doubly open} is bounded, for $\varepsilon$ small enough, by 
$$\binom{q-1}{m}\bar{\varepsilon}^{p+m}.$$
So, now, we need to bound  a sum over $q,m, (\ell_k)_k, (j_k)_k, (f(i))_i, (g(k,i))_{k,i}$ and $(h(k))_{k}$  of
\begin{equation}\label{e:un_terme_arete}
\binom{q-1}{m}\bar{\varepsilon}^{m+q-1 + \sum \ell_k} \1_{\text{Item }1}\1_{\text{Item }2}\1_{\text{Item }4}\1_{\text{Item }5}.
\end{equation}

Using the same techniques as in Section \ref{s:proof-sommable}, we get the following bound for the sum $\bW (q,m,\ell_\cdot,j_\cdot)$ of \eqref{e:un_terme_arete} over $(f(i))_i, (g(k,i))_{k,i}$ and $(h(k))_{k}$:

\begin{align*}
\bW(q,m,\ell_\cdot,j_\cdot) \le
& \binom{q-1}{m}\bar{\varepsilon}^{m}e^{\alpha m}\\
& \frac {1} {\bar{\epsilon}} \exp\Big[2+ \alpha+\beta \bar{\epsilon}^{1/d_1} n\|z\|_1  -\alpha \eta n \bar{\epsilon}^{1/d_1}\Big] \\
& \left(\bar{\epsilon} e^2\exp[2\gamma \bar{\epsilon}^{1/d_1}]K(1)^{d_2}K(\beta \bar{\epsilon}^{1/d_1})K(\gamma  \bar{\epsilon}^{1/d_1})^{(d_1-1)}\right)^{q} \\
& \left(\bar{\epsilon} \exp[2\gamma \bar{\epsilon}^{1/d_1} + 2\alpha ]
e^2K(1)^{d_2}K(\alpha)\big(K(\gamma\bar{\epsilon}^{1/d_1})+K(\alpha-\gamma\bar{\epsilon}^{1/d_1})\big)^{d_1-1}\right)^{\sum_k \ell_k}.
\end{align*}

Thus, summing over $m$, we get

\begin{align*}
\sum_{m=0}^{q-1}\bW(q,m,\ell_\cdot,j_\cdot) \le  
& \frac {1} {\bar{\epsilon}} \exp\Big[2+ \alpha+\beta \bar{\epsilon}^{1/d_1} n\|z\|_1  -\alpha \eta n \bar{\epsilon}^{1/d_1}\Big] \\
& \left((1+\bar{\epsilon}e^\alpha)\bar{\epsilon} e^2\exp[2\gamma \bar{\epsilon}^{1/d_1}]K(1)^{d_2}K(\beta \bar{\epsilon}^{1/d_1})K(\gamma  \bar{\epsilon}^{1/d_1})^{(d_1-1)}\right)^{q} \\
& \left(\bar{\epsilon} \exp[2\gamma \bar{\epsilon}^{1/d_1} + 2 \alpha ]
e^2K(1)^{d_2}K(\alpha)\big(K(\gamma\bar{\epsilon}^{1/d_1})+K(\alpha-\gamma\bar{\epsilon}^{1/d_1})\big)^{d_1-1}\right)^{\sum_k \ell_k}
\end{align*}
which is roughly the same bound as the one obtain in \eqref{e:borneW}. Hence, the end of the proof is quite identical as in Section \ref{s:proof-sommable}. Choosing $\alpha, \beta$ and $\gamma$ (depending on $z$ and $\eta$) such that
\[
\beta = 2 \gamma , \quad e^2 K(1)^{d_2}2^{d_1} \beta^{-1}\gamma^{-(d_1 -1)}= \frac{1}{8d_1} \textrm{ and } \alpha \eta - \beta \|z\|_1 =1
\]
we obtain that, for all $\epsilon>0$ small enough (depending on $z$ and $\eta$),
\[
\mathbb{P} [\bcM(z,\epsilon,\eta,n)] \leq \frac{e^{2+\alpha}}{\bar{\epsilon}} \exp \left[ - \bar{\epsilon}^{1/d_1} n \right].
\]

\subsection{Proof of Theorem \ref{t2} and Proposition \ref{p}}

Theorem \ref{t2} is a direct consequence of Lemmas \ref{l:link_discrete_continous_arete} and \ref{a:reduction_oriented}, with
\begin{equation}
\label{e:expressionC2}
\bC (d_1,d_2) =  \bsigma^\cC_0(\1_{d_1})
\end{equation}
by \eqref{e:ajoutM2}. Moreover, $ \bsigma^\cC_0(\1_{d_1})$ is finite by Lemma \ref{l:encadrement_arete} and the fact that $ \sigma^\cC_0(\1_{d_1}) < \infty$ (see Lemma \ref{l:valuesigmac}).

The first part of Proposition \ref{p}, {\it i.e.}, the relation between the constants $C(d_1,d_2)$ and $\bC(d_1,d_2)$ is a straightforward consequence of Lemma \ref{l:encadrement_arete} considering the definition of the constants $C(d_1,d_2)$ and $\bC(d_1,d_2)$ given by \eqref{e:expressionC} and \eqref{e:expressionC2}. The last fact to check is the equality
\[
C(2,0) = 2\,.
\]
By \eqref{e:expressionC} we know that $C(2,0)= \sigma^\cC_0((1,1))$, the mean directional score in the semi-continuous oriented model associated to site Bernoulli first-passage percolation for $(d_1,d_2) = (2,0)$ as defined by Lemma \ref{l:sigmac} in Section \ref{s:continuous_model}. Since $d_2=0$, this model is in fact totally continuous and oriented, and it is solvable: it is a continuous Poissonization version of the discrete Ulam's problem described in the introduction, introduced by Hammersley \cite{Hammersley}. Hammersley conjectured $ \sigma^\cC_0((1,1))=2$, and this conjecture was proved first by Logan and Shepp and by Vershik and Kerov in 1977, and then in a more probabilistic way by Aldous and Diaconis \cite{AldousDiaconis}  in 1995, using the so-called Hammersley's line process. 

\appendix

\section{Proof of some standard results}

\subsection{Existence of the time constant in Bernoulli site first-passage percolation}
\label{s:proof-cte}

The time constant in Bernoulli site first-passage percolation can be defined for any $z \in \mathbb{R}^d\setminus \{0\}$ as 
\[
\mu_\epsilon (z) = \lim_{n\rightarrow \infty} \frac{T_\epsilon (0, \lfloor nz \rfloor )}{n} \quad \textrm{a.s. and in }L^1.
\]
The proof of this result is an exact copy of its more classical version in Bernoulli bond first-passage percolation, see for instance \cite{Kesten-saint-flour}. When $z\in \mathbb{Z}^d$, this is a straightforward consequence of Kingman ergodic subadditive theorem. The convergence can be extended to any $z\in \mathbb{Q}^d$ by homogeneity. Obtaining the convergence for any $z\in \mathbb{R}^d$ requires a little more work, that is standard but would require a few pages. For completeness of the paper, we make the choice to give an explicit proof of a much simpler result, namely the convergence of the expectations of the rescaled passage times to the time constant. The result has the double advantage to be easy to prove, and to give a rigorous definition of $\mu_\epsilon (z)$ that is sufficient for our study.

\begin{lemma}
\label{l:cte}
For any $z \in \mathbb{R}^d\setminus \{0\}$, the following limit is well defined:
\[
\mu_\epsilon (z) := \lim_{n\rightarrow \infty} \frac{1}{n} \mathbb{E} [T_\epsilon (0, \lfloor nz \rfloor )].
\]
\end{lemma}

\begin{proof}
Fix some $z\in \R^d$ and let $u_n:=\E [T_\epsilon (0, \lfloor nz \rfloor )]$. By triangle inequality, we have
$$T_\epsilon (0, \lfloor (n+m)z \rfloor )\le T_\epsilon (0, \lfloor nz \rfloor )+T_\epsilon ( \lfloor nz \rfloor,\lfloor nz \rfloor+\lfloor mz \rfloor )+T_\epsilon (\lfloor nz \rfloor+\lfloor mz \rfloor, \lfloor (n+m)z \rfloor ).
$$
Thus, taking the expectation and using the invariance by translation, we get
$$u_{n+m}\le u_n+u_m+\E[T_\epsilon (0, \lfloor (n+m)z \rfloor-\lfloor nz \rfloor-\lfloor mz \rfloor )].$$
One can easily check that 
$$\|\lfloor (n+m)z \rfloor-\lfloor nz \rfloor-\lfloor mz \rfloor \|_\infty\le 1$$ 
and so, since the passage time at each site is bounded by 1, we get that 
$$T_\epsilon (0, \lfloor (n+m)z \rfloor-\lfloor nz \rfloor-\lfloor mz \rfloor )\le d.$$
Thus, $(u_n+d)_{n\ge 0}$ is a subadditive sequence and Fekete's Lemma implies that 
$$\mu_\epsilon (z) := \lim_{n\rightarrow \infty} \frac{1}{n} \mathbb{E} [T_\epsilon (0, \lfloor nz \rfloor )]=\lim_{n\rightarrow \infty} \frac{u_n}{n}=\lim_{n\rightarrow \infty} \frac{u_n+d}{n}$$ exists.
\end{proof}

\subsection{Proof of Lemma \ref{l:path-nice}}
\label{s:proof-path-nice}

Let $z \in (0,+\infty)^{d_1} \times \{0\}^{d_2}$ and $n \in \mathbb{N}^*$. From any nice path from $0$ to $\lfloor nz\rfloor^-$ we get a path from $0$ to $\lfloor nz\rfloor$
by adding $d_1$ steps. The travel time of the path increases by at most $d_1$. 
From this observation and by definition of $ \mu_\epsilon(z)$ (see Lemma \ref{l:cte}) one gets
\[
\mu_\epsilon(z) = \lim_{n \to \infty} \frac 1 n \E[T_\epsilon(0,\lfloor nz \rfloor)]
 \le \liminf_{n\to \infty} \frac 1 n \E[T^{\text{nice}}_\epsilon(0,\lfloor nz\rfloor)].
\]

To prove Lemma \ref{l:path-nice}, it remains to prove that $\mu_\epsilon(z) \geq  \limsup_n \E[T^{\text{nice}}_\epsilon(0,\lfloor nz\rfloor)] /n$.
Let $\epsilon>0$.
Let $\eta \in (0,1)$.
By definition of $ \mu_\epsilon(z)$ (see Lemma \ref{l:cte}), we can fix $p\ge 1$ such that
\[
\frac 1 p \E[T_\epsilon(0, \lfloor pz \rfloor )] \le \mu_\epsilon(z) + \eta
\]
and
\begin{equation}\label{e:pgrand2}
(1-\eta) pz \preceq \lfloor pz \rfloor
\end{equation}
and 
\begin{equation}\label{e:pgrand3}
\frac{d_1}p \le \eta.
\end{equation}
For $M \ge 1$ and $x \prec y$ in $\Z^{d_1} \times \{0\}^{d_2}$ we say that $\pi=(w(0),\dots,w(n))$ is a $M$-path from $x$ to $y$ if 
$\pi$ is a path from $x$ to $y$ and, for all $i \in \{0,\dots,n\}$,
\[
w(i) \preceq y + M \1_{d_1}.
\]
We denote by $T_\epsilon^M(x,y)$ the infimum of travel times $\tau_\epsilon(\pi)$ over all $M$-paths $\pi$ from $x$ to $y$.
By dominated convergence, we get
\[
\lim_{M \to +\infty} \E[T_\epsilon^M(0, \lfloor pz \rfloor )] = \E[T_\epsilon(0, \lfloor pz \rfloor )].
\]
Therefore we can fix $M$ such that
\[
\frac 1 p \E[T_\epsilon^M(0, \lfloor pz \rfloor )] \le \mu_\epsilon(z)+2\eta.
\]
Let now $n$ be a large integer. Let $q$ be the largest integer such that
\begin{equation}\label{e:q}
q \lfloor pz \rfloor + M\1_{d_1} \preceq \lfloor nz \rfloor^-.
\end{equation}
Define $r \in (\N^*)^{d_1}\times\{0\}^{d_2}$ by
\[
\lfloor nz \rfloor^- = q \lfloor pz \rfloor + r.
\]
Gluing a $M$-path from $0$ to $\lfloor pz \rfloor$,
a $M$-path from $\lfloor pz \rfloor$ to $2\lfloor pz \rfloor$,
and so on until a $M$-path from $(q-1)\lfloor pz \rfloor$ to $q\lfloor pz \rfloor$ 
and then any shortest path (in number of edges) from $q\lfloor pz \rfloor$ to $\lfloor nz\rfloor^-$,
we get a nice path from $0$ to $\lfloor nz\rfloor^-$.
Optimizing on the $M$-paths, taking expectation, using stationarity and bounding the travel time of the last part of the path
by its length, we get
\[
\E[T^{\text{nice}}_\epsilon(0,\lfloor nz\rfloor)] \le q \E[T_\epsilon^M(0, \lfloor pz \rfloor )] + \|r\|_1.
\]
Thus
\begin{align}
\frac 1 n \E[T^{\text{nice}}_\epsilon(0,\lfloor nz\rfloor)]
 &  \le \frac q n \E[T_\epsilon^M(0, \lfloor pz \rfloor )] + \frac{\|r\|_1}n \nonumber \\
 &  \le \frac {pq} n (\mu_\epsilon(z)+2\eta) + \frac{\|r\|_1}n. \label{e:majETnice}
\end{align}
The desired inequality $\mu_\epsilon(z) \geq  \limsup_n \E[T^{\text{nice}}_\epsilon(0,\lfloor nz\rfloor)] /n$ will follow from \eqref{e:majETnice} as expected, but proving this implication requires a few lines. From \eqref{e:q} we get
\begin{equation}\label{e:blop}
q = \min_{i \in \{1,\dots,d_1\}} \frac{ \lfloor nz \rfloor_i - 1 - M}{\lfloor pz \rfloor_i}.
\end{equation}
Using \eqref{e:pgrand2} we deduce, for $n$ large enough, 
\begin{equation}\label{e:majq}
q \le \min_{i \in \{1,\dots,d_1\}} \frac{ nz_i }{(1-\eta)pz_i} = \frac{n}{p(1-\eta)}.
\end{equation}
From \eqref{e:blop} we also get, for $n$ large enough,
\[
q \ge \min_{i \in \{1,\dots,d_1\}} \frac{ nz_i  - 2 - M}{pz_i} = \frac n p - \max_{i\in\{1,\dots,d_1\}} \frac{2+M}{p z_i}
\]
and then
\begin{align*}
r & = \lfloor nz \rfloor^- - q\lfloor pz \rfloor \\
& \preceq nz - \frac n p \lfloor pz \rfloor + \max_{i\in\{1,\dots,d_1\}} \frac{2+M}{p z_i}\lfloor pz \rfloor \\
& \preceq nz - \frac n p pz + \frac n p \1_{d_1} + \max_{i\in\{1,\dots,d_1\}} \frac{2+M}{p z_i}\lfloor pz \rfloor \\
& \preceq \frac n p \1_{d_1} + \max_{i\in\{1,\dots,d_1\}} \frac{2+M}{p z_i}\lfloor pz \rfloor
\end{align*}
and thus
\begin{equation}\label{e:majr1}
\|r\|_1 \le \frac {nd_1}p + \max_{i\in\{1,\dots,d_1\}} \frac{2+M}{p z_i} \|\lfloor pz \rfloor\|_1.
\end{equation}
From \eqref{e:majETnice}, \eqref{e:majq}, \eqref{e:majr1} and \eqref{e:pgrand3} we get
\[
\limsup_{n \to \infty} \frac 1 n \E[T^{\text{nice}}_\epsilon(0,\lfloor nz\rfloor)] \le \frac{\mu_\epsilon(z)+2\eta}{1-\eta}  + \eta.
\]
As this holds for any $\eta>0$ we get
\[
\limsup_{n\to \infty} \frac 1 n \E[T^{\text{nice}}_\epsilon(0,\lfloor nz\rfloor)] \le \mu_\epsilon(z).
\]
This proves the lemma.
 \qed

\paragraph{Aknowledgements :} Research was partially supported by the ANR project PPPP (ANR-16-CE40-0016) and the Labex MME-DII (ANR 11-LBX-0023-01).

\bibliographystyle{plain}

\end{document}